\documentclass[a4paper,reqno]{amsart}
\usepackage{mathrsfs}
\usepackage{syntonly}
\usepackage{amsmath}
\usepackage{amsthm}
\usepackage{amsfonts}
\usepackage{amssymb}
\usepackage{latexsym}
\usepackage{amscd,amssymb,amsopn,amsmath,amsthm,graphics,amsfonts,mathrsfs,accents,enumerate,verbatim,calc}
\usepackage[dvips]{graphicx}
\usepackage[colorlinks=true,linkcolor=blue,citecolor=blue]{hyperref}
\input xy
\xyoption{all}
\def\del{\delta}

\usepackage{color}

\numberwithin{equation}{section}

\theoremstyle{plain}

\newtheorem{thm}{Theorem}[section]

\newtheorem{lem}[thm]{Lemma}
\newtheorem{prop}[thm]{Proposition}

\newtheorem{defn}[thm]{Definition}
\newtheorem{exm}[thm]{Example}

\theoremstyle{remark}
\newtheorem{rem}[thm]{Remark}

\newdir{ >}{{}*!/-5pt/\dir{>}}

\renewcommand{\mod}{\operatorname{mod}\nolimits}

\newcommand{\add}{\operatorname{add}\nolimits}
\newcommand{\Fac}{\operatorname{Fac}\nolimits}

\newcommand{\Hom}{\operatorname{Hom}\nolimits}

\newcommand{\End}{\operatorname{End}\nolimits}

\newcommand{\Ext}{\operatorname{Ext}\nolimits}

\newcommand{\pd}{\operatorname{pd}\nolimits}

\newcommand{\Cone}{\operatorname{Cone}\nolimits}
\newcommand{\CoCone}{\operatorname{CoCone}\nolimits}

\newcommand{\B}{\mathcal B}
\newcommand{\uB}{\underline{\B}}
\newcommand{\oB}{\overline{\B}}

\newcommand{\h}{\mathcal H}

\newcommand{\I}{\mathcal I}

\newcommand{\R}{\mathcal R}

\newcommand{\C}{\mathcal C}

\newcommand{\EE}{\mathbb E}

\newcommand{\svecv}[2]{\left(\begin{smallmatrix}
      #1 \\
      #2
    \end{smallmatrix}\right)}

\newcommand{\svech}[2]{\left(\begin{smallmatrix}
      #1 & #2
\end{smallmatrix}\right)}

\def\Ab{\mathsf{Ab}}
\renewcommand{\emph}{\textit}
\renewcommand{\phi}{\varphi}

\newcommand{\pr}{\mathsf{pr}\hspace{.01in}}

\begin{document}

\title{Relative rigid objects in extriangulated categories}\footnote{
The first author is supported by the Fundamental Research Funds for the Central Universities (Grants No.2682019CX51). The second author is supported by the Hunan Provincial Natural Science Foundation of China (Grants No.2018JJ3205) and the NSF of China (Grants No.11671221).}
\author{Yu Liu and Panyue Zhou}
\address{School of Mathematics, Southwest Jiaotong University, 610031, Chengdu, Sichuan, People's Republic of China}
\email{liuyu86@swjtu.edu.cn}
\address{College of Mathematics, Hunan Institute of Science and Technology, 414006, Yueyang, Hunan, People's Republic of China}
\email{panyuezhou@163.com}
\thanks{The authors wish to thank Professor Bin Zhu for his helpful advices.}
\begin{abstract}
In this paper, we study a close relationship between relative cluster tilting theory in extriangulated categories and $\tau$-tilting
theory in module categories. Our main results show that relative rigid objects are in bijection with $\tau$-rigid pairs, and also relative maximal rigid objects with
support $\tau$-tilting pairs under some assumptions.
These results generalize their work by Adachi-Iyama-Reiten, Yang-Zhu and Fu-Geng-Liu.
Finally, we introduce mutation of relative maximal rigid objects and show that any basic relative
almost maximal rigid object has exactly two non-isomorphic indecomposable complements. 
\end{abstract}
\maketitle

\section{Introduction}

In \cite{AIR}, Adachi, Iyama and Reiten introduced a generalization of classical tilting theory, which is called $\tau$-tilting theory.
They proved that for a $2$-Calabi-Yau triangulated category $\C$ with a cluster tilting object $T$, there exists a
 bijection between the basic cluster tilting objects in $\C$ and the basic support $\tau$-tilting modules in $\mod\End_\C(T)^{\textrm{op}}$.
Note that any cluster tilting object is maximal rigid in triangulated category, but
the converse is not true in general. Chang-Zhang-Zhu \cite{CZZ} and Liu-Xie \cite{LX}
showed that the similar version of the above bijection is also valid for a $2$-Calabi-Yau triangulated category with a rigid object. When the triangulated category $\C$ is not $2$-Calabi-Yau, but with a cluster tilting object $T$, the Adachi-Iyama-Reiten's bijection does not exist, see \cite[Example 2.15]{YZ}.
It is then reasonable to find a class of objects in $\C$ which correspond to support $\tau$-tilting modules in $\mod\End_\C(T)^{\textrm{op}}$ bijectively in a more general setting. For this purpose, Yang and Zhu \cite{YZ} introduced the notion of relative cluster tilting objects in a triangulated category $\C$ with a cluster tilting object, which is a generalization of cluster tilting objects. Let $\C$ be a triangulated category with the shift functor $[1]$ and a cluster tilting object $T$.
They introduced the notion of $T[1]$-cluster tilting objects and established a one-to-one correspondence between the basic $T[1]$-cluster tilting objects of $\C$ and the basic support $\tau$-tilting modules in $\mod\End_\C(T)^{\textrm{op}}$.
This bijection was generalized  by Fu, Geng and Liu \cite{FGL} recently to a triangulated category $\C$ with a rigid object. Let $R\in \C$ be a rigid object with endomorphism algebra $\Gamma$. They introduced the notion of the $R[1]$-rigid objects in the finitely presented subcategory $\pr R $ of $\C$ and showed that there exists a bijection between the set of basic $R[1]$-rigid objects in $\pr R$ and the set of basic $\tau$-rigid pairs of $\Gamma$-modules, which induces a one-to-one correspondence between the set of basic maximal $R[1]$-rigid objects with respect to $\pr R$ and the set of basic support $\tau$-tilting $\Gamma$-modules.

Recently, the notion of an extriangulated category was introduced by Nakaoka and Palu in \cite[Definition 2.12]{NP}, which is a simultaneous generalization of exact category and triangulated category. For extriangulated
categories which are neither exact categories nor triangulated categories, see \cite[Proposition 3.30]{NP} and \cite[Example 4.14]{ZZ}. A natural question is do such bijections exist when we consider an extriangulated category instead of a triangulated category.
Motivated by this, we study the similar problems in \cite{FGL} and \cite{YZ} in an extriangulated category.

In this paper, let $k$ be a field and $\B$ be a Krull-Schmidt, Hom-finite, $k$-linear extriangulated category with enough projectives $\mathcal P$ and enough injectives $\mathcal I$, and $R$ a basic rigid object
of $\B$ which does not have any projective direct summand. We denote by $\R=(\add R) \vee \mathcal P$ the smallest subcategory of $\B$ containing all the direct sums of objects in $\add R$ and $\mathcal P$. We will introduce the relative rigid objects (for convenience, we also call them $\R$-rigid objects, see Definition \ref{d1}) and
relative maximal rigid objects (for convenience, we also call them maximal $\R$-rigid objects, see Definition \ref{d2}) in $\B$. 

Let $R=\bigoplus\limits^n_{i=1}R_i$ where $R_i$ is indecomposable. Let $\Omega R_i$ be the object induced by the following $\EE$-triangle $$\xymatrix{\Omega R_i \ar[r] &P_i \ar[r]^{p_i} &R_i\ar@{-->}[r] &}$$
where $P_i\in \mathcal P$ and $p_i$ is a right minimal, we denote $\Omega R:=\bigoplus\limits^n_{i=1}\Omega R_i$ and get the following functor $$\Hom_{\B/\mathcal P}(\Omega R,-):\B\to \mod\Gamma, ~{\rm where}~ \Gamma:=\End_{\B/\mathcal P}(\Omega R)$$
Let $\h=\{ \text{ } X\in \B \text{ }|\text{ } \textrm{there exists an} \text{ } \EE\text{-triangle } \xymatrix@C=0.8cm@R0.6cm{ X \ar[r] &R' \ar[r] &R'' \ar@{-->}[r] &} \text{, }R',R''\in \add R \text{ } \}$, we show the following (see Theorem \ref{main}):
\begin{thm}
\begin{itemize}
\item[(a)] Let $X\in \h$. Then $X$ is $\R$-rigid if and only if $\Hom_{\B/\mathcal P}(\Omega R,X)$ is $\tau$-rigid.
\item[(b)] $\Hom_{\B/\mathcal P}(\Omega R,-)$ yields a bijection between the set of isomorphism classes of basic $\R$-rigid objects in $\h$ which have no direct summands in $\mathcal P$ and the set of isomorphism classes of basic $\tau$-rigid pairs of $\Gamma$-modules.
\item[(c)] $\Hom_{\B/\mathcal P}(\Omega R,-)$ yields a bijection between the set of isomorphism classes of basic maximal $\R$-rigid objects in $\h$ which have no direct summands in $\mathcal P$ and the set of isomorphism classes of basic support $\tau$-tilting pairs of $\Gamma$-modules.
\end{itemize}
\end{thm}
Since tilting modules are faithful support $\tau$-tilting modules, using the
correspondence with relative rigid objects, we give an equivalent characterization on tilting modules, see Theorem \ref{main1}.
We also introduce mutation of relative maximal rigid objects and show that any basic relative almost
maximal rigid object has exactly two non-isomorphic indecomposable complements, see Theorem \ref{main2}.
Finally, we give an example illustrating our these results, see Example \ref{ex1}.

\section{Preliminaries}
\subsection{Extriangulated categories}
Let us briefly recall the definition and basic properties of extriangulated categories from \cite{NP}. Throughout this paper, we assume that $\B$ is an additive category.

\begin{defn}\cite[Definition 2.1]{NP}
Suppose that $\B$ is equipped with an additive bifunctor $$\mathbb{E}\colon\B^\mathrm{op}\times\B\to\Ab,$$ where $\Ab$ is the category of abelian groups. For any pair of objects $A,C\in\B$, an element $\delta\in\mathbb{E}(C,A)$ is called an {\it $\mathbb{E}$-extension}. Thus formally, an $\EE$-extension is a triplet $(A,\delta,C)$.
For any $A,C\in\C$, the zero element $0\in\EE(C,A)$ is called the \emph{spilt $\EE$-extension}.

Let $\delta\in\mathbb{E}(C,A)$ be any $\mathbb{E}$-extension. By the functoriality, for any $a\in\B(A,A^{\prime})$ and $c\in\B(C^{\prime},C)$, we have $\mathbb{E}$-extensions
\[ \mathbb{E}(C,a)(\delta)\in\mathbb{E}(C,A^{\prime})\ \ \text{and}\ \ \mathbb{E}(c,A)(\delta)\in\mathbb{E}(C^{\prime},A). \]
We denote them by $a_{\ast}\delta$ and $c^{\ast}\delta$.
In this terminology, we have
\[ \mathbb{E}(c,a)(\delta)=c^{\ast} a_{\ast}\delta=a_{\ast} c^{\ast}\delta \]
in $\mathbb{E}(C^{\prime},A^{\prime})$.
\end{defn}

\begin{defn}\cite[Definition 2.3]{NP}
Let $\delta\in\mathbb{E}(C,A)$ and $\delta^{\prime}\in\mathbb{E}(C^{\prime},A^{\prime})$ be two pair of $\mathbb{E}$-extensions. A {\it morphism} $(a,c)\colon\delta\to\delta^{\prime}$ of $\mathbb{E}$-extensions is a pair of morphisms $a\in\B(A,A^{\prime})$ and $c\in\B(C,C^{\prime})$ in $\B$, satisfying the equality
\[ a_{\ast}\delta=c^{\ast}\delta^{\prime}. \]
We simply denote it as $(a,c)\colon\delta\to\delta^{\prime}$.
\end{defn}

\begin{defn}\cite[Definition 2.6]{NP}
Let $\delta=(A,\delta,C)$ and $\delta^{\prime}=(A^{\prime},\delta^{\prime},C^{\prime})$ be any pair of $\mathbb{E}$-extensions. Let
\[ C\xrightarrow{~\iota_C~}C\oplus C^{\prime}\xleftarrow{~\iota_{C^{\prime}}~}C^{\prime} \]
and
\[ A\xrightarrow{~p_A~}A\oplus A^{\prime}\xleftarrow{~p_{A^{\prime}}~}A^{\prime} \]
be coproduct and product in $\B$, respectively. Remark that, by the additivity of $\mathbb{E}$, we have a natural isomorphism
\[ \mathbb{E}(C\oplus C^{\prime},A\oplus A^{\prime})\simeq \mathbb{E}(C,A)\oplus\mathbb{E}(C,A^{\prime})\oplus\mathbb{E}(C^{\prime},A)\oplus\mathbb{E}(C^{\prime},A^{\prime}). \]

Let $\delta\oplus\delta^{\prime}\in\mathbb{E}(C\oplus C^{\prime},A\oplus A^{\prime})$ be the element corresponding to $(\delta,0,0,\delta^{\prime})$ through this isomorphism. This is the unique element which satisfies
$$
\mathbb{E}(\iota_C,p_A)(\delta\oplus\delta^{\prime})=\delta,\ \mathbb{E}(\iota_C,p_{A^{\prime}})(\delta\oplus\delta^{\prime})=0,\
\mathbb{E}(\iota_{C^{\prime}},p_A)(\delta\oplus\delta^{\prime})=0,\ \mathbb{E}(\iota_{C^{\prime}},p_{A^{\prime}})(\delta\oplus\delta^{\prime})=\delta^{\prime}.
$$
\end{defn}

\begin{defn}\cite[Definition 2.7]{NP}
Let $A,C\in\B$ be any pair of objects. Two sequences of morphisms in $\B$
\[ A\overset{x}{\longrightarrow}B\overset{y}{\longrightarrow}C\ \ \text{and}\ \ A\overset{x^{\prime}}{\longrightarrow}B^{\prime}\overset{y^{\prime}}{\longrightarrow}C \]
are said to be {\it equivalent} if there exists an isomorphism $b\in\B(B,B^{\prime})$ which makes the following diagram commutative.
\[
\xy
(-16,0)*+{A}="0";
(3,0)*+{}="1";
(0,8)*+{B}="2";
(0,-8)*+{B^{\prime}}="4";
(-3,0)*+{}="5";
(16,0)*+{C}="6";
{\ar^{x} "0";"2"};
{\ar^{y} "2";"6"};
{\ar_{x^{\prime}} "0";"4"};
{\ar_{y^{\prime}} "4";"6"};
{\ar^{b}_{\simeq} "2";"4"};
{\ar@{}|{} "0";"1"};
{\ar@{}|{} "5";"6"};
\endxy
\]

We denote the equivalence class of $A\overset{x}{\longrightarrow}B\overset{y}{\longrightarrow}C$ by $[A\overset{x}{\longrightarrow}B\overset{y}{\longrightarrow}C]$.
\end{defn}

\begin{defn}
$\ \ $
\begin{enumerate}
\item[\rm{(1)}] For any $A,C\in\B$, we denote as
\[ 0=[A\overset{\Big[\raise1ex\hbox{\leavevmode\vtop{\baselineskip-8ex \lineskip1ex \ialign{#\crcr{$\scriptstyle{1}$}\crcr{$\scriptstyle{0}$}\crcr}}}\Big]}{\longrightarrow}A\oplus C\overset{[0\ 1]}{\longrightarrow}C]. \]

\item[\rm{(2)}] For any $[A\overset{x}{\longrightarrow}B\overset{y}{\longrightarrow}C]$ and $[A^{\prime}\overset{x^{\prime}}{\longrightarrow}B^{\prime}\overset{y^{\prime}}{\longrightarrow}C^{\prime}]$, we denote as
\[ [A\overset{x}{\longrightarrow}B\overset{y}{\longrightarrow}C]\oplus [A^{\prime}\overset{x^{\prime}}{\longrightarrow}B^{\prime}\overset{y^{\prime}}{\longrightarrow}C^{\prime}]=[A\oplus A^{\prime}\overset{x\oplus x^{\prime}}{\longrightarrow}B\oplus B^{\prime}\overset{y\oplus y^{\prime}}{\longrightarrow}C\oplus C^{\prime}]. \]
\end{enumerate}
\end{defn}

\begin{defn}\cite[Definition 2.9]{NP}
Let $\mathfrak{s}$ be a correspondence which associates an equivalence class $\mathfrak{s}(\delta)=[A\overset{x}{\longrightarrow}B\overset{y}{\longrightarrow}C]$ to any $\mathbb{E}$-extension $\delta\in\mathbb{E}(C,A)$. This $\mathfrak{s}$ is called a {\it realization} of $\mathbb{E}$, if it satisfies the following condition.
\begin{itemize}
\item Let $\delta\in\mathbb{E}(C,A)$ and $\delta^{\prime}\in\mathbb{E}(C^{\prime},A^{\prime})$ be any pair of $\mathbb{E}$-extensions, with
\[\mathfrak{s}(\delta)=[A\overset{x}{\longrightarrow}B\overset{y}{\longrightarrow}C]\text{ and } \mathfrak{s}(\delta^{\prime})=[A^{\prime}\overset{x^{\prime}}{\longrightarrow}B^{\prime}\overset{y^{\prime}}{\longrightarrow}C^{\prime}].\]
Then, for any morphism $(a,c)\colon\delta\to\delta^{\prime}$, there exists $b\in\B(B,B^{\prime})$ which makes the following diagram commutative.
$$
\xy
(-12,6)*+{A}="0";
(0,6)*+{B}="2";
(12,6)*+{C}="4";
(-12,-6)*+{A^{\prime}}="10";
(0,-6)*+{B^{\prime}}="12";
(12,-6)*+{C^{\prime}}="14";
{\ar^{x} "0";"2"};
{\ar^{y} "2";"4"};
{\ar_{a} "0";"10"};
{\ar^{b} "2";"12"};
{\ar^{c} "4";"14"};
{\ar^{x^{\prime}} "10";"12"};
{\ar^{y^{\prime}} "12";"14"};
{\ar@{}|{} "0";"12"};
{\ar@{}|{} "2";"14"};
\endxy
$$
\end{itemize}
In this case, we say that the sequence $A\overset{x}{\longrightarrow}B\overset{y}{\longrightarrow}C$ {\it realizes} $\delta$, whenever it satisfies $\mathfrak{s}(\delta)=[A\overset{x}{\longrightarrow}B\overset{y}{\longrightarrow}C]$.
In the above situation, we also say that the triplet $(a,b,c)$ {\it realizes} $(a,c)$.
\end{defn}

\begin{defn}\cite[Definition 2.10]{NP}
Let $\B,\mathbb{E}$ be as above. A realization of $\mathbb{E}$ is said to be {\it additive}, if it satisfies the following conditions.
\begin{itemize}
\item[{\rm (i)}] For any $A,C\in\B$, the split $\mathbb{E}$-extension $0\in\mathbb{E}(C,A)$ satisfies
\[ \mathfrak{s}(0)=0. \]
\item[{\rm (ii)}] For any pair of $\mathbb{E}$-extensions $\delta\in\mathbb{E}(C,A)$ and $\delta^{\prime}\in\mathbb{E}(C^{\prime},A^{\prime})$, we have
\[ \mathfrak{s}(\delta\oplus\delta^{\prime})=\mathfrak{s}(\delta)\oplus\mathfrak{s}(\delta^{\prime}). \]
\end{itemize}
\end{defn}

\begin{defn}\cite[Definition 2.12]{NP}
A triplet $(\B,\mathbb{E},\mathfrak{s})$ is called an {\it extriangulated category} if it satisfies the following conditions.
\begin{itemize}
\item[{\rm (ET1)}] $\mathbb{E}\colon\B^{\mathrm{op}}\times\B\to\Ab$ is an additive bifunctor.
\item[{\rm (ET2)}] $\mathfrak{s}$ is an additive realization of $\mathbb{E}$.
\item[{\rm (ET3)}] Let $\delta\in\mathbb{E}(C,A)$ and $\delta^{\prime}\in\mathbb{E}(C^{\prime},A^{\prime})$ be any pair of $\mathbb{E}$-extensions, realized as
\[ \mathfrak{s}(\delta)=[A\overset{x}{\longrightarrow}B\overset{y}{\longrightarrow}C],\ \ \mathfrak{s}(\delta^{\prime})=[A^{\prime}\overset{x^{\prime}}{\longrightarrow}B^{\prime}\overset{y^{\prime}}{\longrightarrow}C^{\prime}]. \]
For any commutative square
$$
\xy
(-12,6)*+{A}="0";
(0,6)*+{B}="2";
(12,6)*+{C}="4";
(-12,-6)*+{A^{\prime}}="10";
(0,-6)*+{B^{\prime}}="12";
(12,-6)*+{C^{\prime}}="14";
{\ar^{x} "0";"2"};
{\ar^{y} "2";"4"};
{\ar_{a} "0";"10"};
{\ar^{b} "2";"12"};
{\ar^{x^{\prime}} "10";"12"};
{\ar^{y^{\prime}} "12";"14"};
{\ar@{}|{} "0";"12"};
\endxy
$$
in $\B$, there exists a morphism $(a,c)\colon\delta\to\delta^{\prime}$ satisfying $cy=y^{\prime}b$.
\item[{\rm (ET3)$^{\mathrm{op}}$}] Dual of {\rm (ET3)}.
\item[{\rm (ET4)}] Let $\delta\in\mathbb{E}(D,A)$ and $\delta^{\prime}\in\mathbb{E}(F,B)$ be $\mathbb{E}$-extensions realized by
\[ A\overset{f}{\longrightarrow}B\overset{f^{\prime}}{\longrightarrow}D\ \ \text{and}\ \ B\overset{g}{\longrightarrow}C\overset{g^{\prime}}{\longrightarrow}F \]
respectively. Then there exist an object $E\in\B$, a commutative diagram
$$
\xy
(-21,7)*+{A}="0";
(-7,7)*+{B}="2";
(7,7)*+{D}="4";
(-21,-7)*+{A}="10";
(-7,-7)*+{C}="12";
(7,-7)*+{E}="14";
(-7,-21)*+{F}="22";
(7,-21)*+{F}="24";
{\ar^{f} "0";"2"};
{\ar^{f^{\prime}} "2";"4"};
{\ar@{=} "0";"10"};
{\ar_{g} "2";"12"};
{\ar^{d} "4";"14"};
{\ar^{h} "10";"12"};
{\ar^{h^{\prime}} "12";"14"};
{\ar_{g^{\prime}} "12";"22"};
{\ar^{e} "14";"24"};
{\ar@{=} "22";"24"};
{\ar@{}|{} "0";"12"};
{\ar@{}|{} "2";"14"};
{\ar@{}|{} "12";"24"};
\endxy
$$
in $\B$, and an $\mathbb{E}$-extension $\delta^{\prime\prime}\in\mathbb{E}(E,A)$ realized by $A\overset{h}{\longrightarrow}C\overset{h^{\prime}}{\longrightarrow}E$, which satisfy the following compatibilities.
\begin{itemize}
\item[{\rm (i)}] $D\overset{d}{\longrightarrow}E\overset{e}{\longrightarrow}F$ realizes $f^{\prime}_{\ast}\delta^{\prime}$,
\item[{\rm (ii)}] $d^{\ast}\delta^{\prime\prime}=\delta$,

\item[{\rm (iii)}] $f_{\ast}\delta^{\prime\prime}=e^{\ast}\delta^{\prime}$.
\end{itemize}

\item[{\rm (ET4)$^{\mathrm{op}}$}] Dual of {\rm (ET4)}.
\end{itemize}
\end{defn}

\begin{rem}
Note that both exact categories and triangulated categories are extriangulated categories, see \cite[Example 2.13]{NP} and extension closed subcategories of extriangulated categories are
again extriangulated, see \cite[Remark 2.18]{NP}. Moreover, there exist extriangulated categories which
are neither exact categories nor triangulated categories, see \cite[Proposition 3.30]{NP} and \cite[Example 4.14]{ZZ}.
\end{rem}

We will use the following terminology.
\begin{defn}{\cite{NP}}
Let $(\B,\EE,\mathfrak{s})$ be an extriangulated category.
\begin{itemize}
\item[(1)] A sequence $A\xrightarrow{~x~}B\xrightarrow{~y~}C$ is called a {\it conflation} if it realizes some $\EE$-extension $\del\in\EE(C,A)$.\\
    In this case, $x$ is called an {\it inflation} and $y$ is called a {\it deflation}.

\item[(2)] If a conflation  $A\xrightarrow{~x~}B\xrightarrow{~y~}C$ realizes $\delta\in\mathbb{E}(C,A)$, we call the pair $( A\xrightarrow{~x~}B\xrightarrow{~y~}C,\delta)$ an {\it $\EE$-triangle}, and write it in the following way.
$$A\overset{x}{\longrightarrow}B\overset{y}{\longrightarrow}C\overset{\delta}{\dashrightarrow}$$
We usually do not write this $``\delta"$ if it is not used in the argument.
\item[(3)] Let $A\overset{x}{\longrightarrow}B\overset{y}{\longrightarrow}C\overset{\delta}{\dashrightarrow}$ and $A^{\prime}\overset{x^{\prime}}{\longrightarrow}B^{\prime}\overset{y^{\prime}}{\longrightarrow}C^{\prime}\overset{\delta^{\prime}}{\dashrightarrow}$ be any pair of $\EE$-triangles. If a triplet $(a,b,c)$ realizes $(a,c)\colon\delta\to\delta^{\prime}$, then we write it as
$$\xymatrix{
A \ar[r]^x \ar[d]^a & B\ar[r]^y \ar[d]^{b} & C\ar@{-->}[r]^{\del}\ar[d]^c&\\
A'\ar[r]^{x'} & B' \ar[r]^{y'} & C'\ar@{-->}[r]^{\del'} &}$$
and call $(a,b,c)$ a {\it morphism of $\EE$-triangles}.

\item[(4)] An object $P\in\B$ is called {\it projective} if
for any $\EE$-triangle $A\overset{x}{\longrightarrow}B\overset{y}{\longrightarrow}C\overset{\delta}{\dashrightarrow}$ and any morphism $c\in\B(P,C)$, there exists $b\in\B(P,B)$ satisfying $yb=c$.
We denote the subcategory of projective objects by $\mathcal P\subseteq\B$. Dually, the subcategory of injective objects is denoted by $\I\subseteq\B$.

\item[(5)] We say that $\B$ {\it has enough projective objects} if
for any object $C\in\B$, there exists an $\EE$-triangle
$A\overset{x}{\longrightarrow}P\overset{y}{\longrightarrow}C\overset{\delta}{\dashrightarrow}$
satisfying $P\in\mathcal P$. Dually we can define $\B$ {\it has enough injective objects}.
\end{itemize}
\end{defn}

By \cite[Corollary 3.5]{NP}, we give the following useful remark, which will be used in the sequel.

\begin{rem}\label{useful}
Let $\xymatrix{A\ar[r]^a &B \ar[r]^b &C \ar@{-->}[r] &}$ and $\xymatrix{X\ar[r]^x &Y \ar[r]^y &Z \ar@{-->}[r] &}$ be two $\EE$-triangles. Then
\begin{itemize}
\item In the following commutative diagram
$$\xymatrix{
X\ar[r]^x \ar[d]_f &Y \ar[d]^g \ar[r]^y &Z \ar[d]^h \ar@{-->}[r] &\\
A\ar[r]^a &B \ar[r]^b &C \ar@{-->}[r] &}
$$
$f$ factors through $x$ if and only if $h$ factors through $b$.


\end{itemize}
\end{rem}

\subsection{Rigid objects and cluster tilting objects}
Let $(\B,\mathbb{E},\mathfrak{s})$ be an extriangulated category with enough projectives $\mathcal P$ and enough injectives $\mathcal I$.

\begin{defn}
Let $\B'$ and $\B''$ be two subcategories of $\B$.
\begin{itemize}
\item[(a)] Denote by $\CoCone(\B',\B'')$ the subcategory
$$\{ \text{ } X\in \B \text{ }|\text{ } ~\textrm{there exists an}~ \text{ } \EE\text{-triangle } \xymatrix@C=0.8cm@R0.6cm{ X \ar[r] &B' \ar[r] &B'' \ar@{-->}[r] &} \text{, }B'\in \B' \text{ and }B''\in \B'' \text{ } \};$$
\item[(b)] Denote by $\Cone(\B',\B'')$ the subcategory
$$\{ \text{ } X\in \B \text{ }|\text{ } ~\textrm{there exists an}~ \text{ } \EE\text{-triangle } \xymatrix@C=0.8cm@R0.6cm{B' \ar[r] &B'' \ar[r] &X \ar@{-->}[r] &} \text{, }B'\in \B' \text{ and }B''\in \B'' \text{ } \};$$
\item[(c)] Let $\Omega^0 \B'=\B'$ and $\Omega \B'=\CoCone(\mathcal P,\B')$, then we can define $\Omega^i \B'$ inductively:
$$\Omega^i \B'=\CoCone(\mathcal P,\Omega^{i-1} \B'),$$
we can define a functor $\Omega: \B \to \B/\mathcal P$ according to the definition above;
\item[(d)] Let $\Sigma^0 \B'=\B'$, $\Sigma \B'=\Cone(\B',\mathcal I)$, then we can define $\Sigma^i \B'$ inductively:
$$\Sigma^i \B'=\Cone(\Sigma^{i-1} \B',\mathcal I),$$
we can define a functor $\Sigma: \B \to \B/\mathcal I$ according to the definition above.
\end{itemize}
We write an object $D$ in the form $\Omega B$ if it admits an $\EE$-triangle $\xymatrix@C=0.8cm@R0.6cm{D \ar[r] &P \ar[r] &B \ar@{-->}[r] &}$ where $P\in \mathcal P$. We write an object $D'$ in the form $\Sigma B'$ if it admits an $\EE$-triangle $\xymatrix@C=0.8cm@R0.6cm{B' \ar[r] &I \ar[r] &D' \ar@{-->}[r] & }$ where $I\in \mathcal I$.
\end{defn}

 Liu and Naokaoka \cite[Proposition 5.2]{LN} defined higher extension groups in an extriangulated category with enough projectives and enough injectives as $\EE^{i+1}(X,Y):=\EE(X,\Sigma^{i}Y)\cong \EE(\Omega^{i}X,Y)$ for $i\geq1$, and they proved the following.

\begin{lem}\cite[Proposition 5.2]{LN}
Let $\xymatrix{A\ar[r]^{x}&B\ar[r]^{y}&C\ar@{-->}[r]^{\delta}&}$ be an $\EE$-triangle. For any object $X\in\B$, there are the following long exact sequences
$$\cdots\rightarrow\EE^{i}(X, A)\xrightarrow{x_{*}}\EE^{i}(X, B)\xrightarrow{y_{*}}\EE^{i}(X, C)\rightarrow\EE^{i+1}(X, A)\xrightarrow{x_{*}}\EE^{i+1}(X, B)\xrightarrow{y_{*}}\cdots~~ (i\geq 1);$$
$$\cdots\rightarrow\EE^{i}(C, X)\xrightarrow{y^{*}}\EE^{i}(B, X)\xrightarrow{x^{*}}\EE^{i}(A, X)\rightarrow\EE^{i+1}(C, X)\xrightarrow{y^{*}}\EE^{i+1}(B, X)\xrightarrow{x^{*}}\cdots~~ (i\geq 1).$$
\end{lem}

\begin{defn}\cite[Definition 5.3]{LN}
Let $(\B,\EE,\mathfrak{s})$ be an extriangulated category with enough projectives and enough injectives.
\begin{itemize}
\item An object $R\in\B$ is called $d$-\emph{rigid} if $\EE^i(R,R)=0$, for any $i\in\{1,2,\cdots, d\}$.
\end{itemize}
\end{defn}

\begin{defn}\cite[Definition 2.7]{ZZ1}
Let $\B$ be an extriangulated category.
\begin{itemize}
\item An object $R\in\B$ is called \emph{rigid} if $\EE(R,R)=0$. In this case, rigid is identical with $1$-rigid.
\item A subcategory $\mathcal R$ of $\B$ is called rigid if $\EE(\mathcal R,\mathcal R)=0$.
\item An object $R\in\B$ is called if \emph{cluster-tilting} if it satisfies
$$\EE(R,M)=0\Leftrightarrow M\in \add R\Leftrightarrow\EE(M,R)=0.$$
\end{itemize}
\end{defn}

\section{Relative rigid objects and $\tau$-rigid pairs}

From this section, let $k$ be a field and $(\B,\mathbb{E},\mathfrak{s})$ be a Krull-Schmidt, Hom-finite, $k$-linear extriangulated category with enough projectives $\mathcal P$ and enough injectives $\mathcal I$.

From now on, we also assume $\B$ satisfies condition (WIC) (\cite[Condition 5.8]{NP}):

\begin{itemize}
\item If we have a deflation $h: A\xrightarrow{~f~} B\xrightarrow{~g~} C$, then $g$ is also a deflation.
\item If we have an inflation $h: A\xrightarrow{~f~} B\xrightarrow{~g~} C$, then $f$ is also an inflation.
\end{itemize}

Note that this condition automatically holds on triangulated categories and Krull-Schmidt exact categories.

By this condition, we can always get right minimal deflations and left minimal inflations.

\subsection{Relative rigid objects} Throughout this paper,
we assume that $R$ is a basic rigid object which does not have direct summands in $\mathcal P$ and $\EE(R,\mathcal P)=0$.  We denote by $\R=(\add R) \vee \mathcal P$ the smallest subcategory of $\B$ containing all the direct sums of objects in $\add R$ and $\mathcal P$. Then $\R$ is a contravariantly finite rigid subcategory of $\B$.

\begin{lem}\label{ind}
Any indecomposable object $R_0\in \add R$ admits an $\EE$-triangle $\xymatrix@C=0.8cm@R0.6cm{\Omega R_0 \ar[r]^{q_0} &P_0 \ar[r]^{p_0} &R_0 \ar@{-->}[r] &}$ where $p_0$ is a right minimal and $\Omega R_0$ is an indecomposable.
\end{lem}

\begin{proof}
Let $\Omega R_0=S_1\oplus S_2$ where $S_1\notin \mathcal P$ is an indecomposable. Then the morphism $\Omega R_0\xrightarrow{q_0} P_0$ can write this form $S_1\oplus S_2\xrightarrow{\svech{q_1}{q_2}} P_0$. The morphism $q_1$ admits an $\EE$-triangle $\xymatrix{S_1 \ar[r]^{q_1} &P_0 \ar[r]^{p_1} &R_1 \ar@{-->}[r] &}$, since $\Ext^1_\B(R_0,P_0)=0$, we get the following commutative diagram:
$$\xymatrix{
S_1 \ar[d]_{\svecv{1}{0}} \ar[r]^{q_1} &P_0 \ar@{=}[d] \ar[r]^{p_1} &R_1 \ar[d]^{r_1} \ar@{-->}[r] &\\
S_1\oplus S_2 \ar[r]^-{\svech{q_1}{q_2}} \ar[d]_-{\svech{1}{0}} &P_0 \ar[r]^{p_0} \ar[d]^a &R_0 \ar[d]^{r_0} \ar@{-->}[r] &\\
S_1 \ar[r]^{q_1} &P_0 \ar[r]^{p_1} &R_1 \ar@{-->}[r] &.
}
$$
Since $\B$ is Krull-Schmidt, $\End_{\B}(R_0)$ is a local ring. Hence either $r_1r_0$ is an isomorphism, or there exists a integer $n>0$ such that $(r_1r_0)^n=0$. If it is the second case, we have the following commutative diagram
$$\xymatrix{
S_1 \ar@{=}[d] \ar[r]^{q_1} &P_0 \ar[d]^{a^{n+1}} \ar[r]^{p_1} &R_1 \ar[d]^{r_0(r_1r_0)^nr_1=0} \ar@{-->}[r] &\\
S_1 \ar[r]^{q_1} &P_0 \ar[r]^{p_1} &R_1 \ar@{-->}[r] &.
}
$$
This implies that $1_{S_1}$ factors through $q_1$, thus $S_1$ is a direct summand of $P_0$ and $S_1\in \mathcal P$, a contradiction. Hence $r_1r_0$ is an isomorphism. We consider the following commutative diagram
$$\xymatrix{
S_1\oplus S_2 \ar[r]^-{\svech{q_1}{q_2}} \ar[d]_-{\svech{1}{0}} &P_0 \ar[r]^{p_0\quad} \ar[d]^b &R_0 \ar[d]^{r_0} \ar@{-->}[r] &\\
S_1 \ar[d]_{\svecv{1}{0}} \ar[r]^{q_1} &P_0 \ar[d]^a \ar[r]^{p_1} &R_1 \ar[d]^{r_1} \ar@{-->}[r] &\\
S_1\oplus S_2 \ar[r]^-{\svech{q_1}{q_2}} &P_0 \ar[r]^{p_0\quad} &R_0 \ar@{-->}[r] &.
}
$$
Since $p_0$ is right minimal, we get that $ab$ is an isomorphism. Hence $S_1\oplus S_2\simeq S_1$, which implies $\Omega R_0=S_1$ is an indecomposable.
\end{proof}

Let $R=\bigoplus\limits^n_{i=1}R_i$ where $R_i$ is indecomposable. In this paper, we denote $\Omega R:=\bigoplus\limits^n_{i=1}\Omega R_i$ where $\Omega R_i$ is the indecomposable object determined by Lemma \ref{ind}.

We denote $\B'/\mathcal P$ by $\underline \B'$ if $\mathcal P\subseteq \B'\subseteq \B$. For any morphism $f\colon A\to B$ in $\B$, we denote by $\underline{f}$ the image of $f$ under the natural quotient functor $\B\to \uB$. We denote $\B'/\R$ by $\overline {\B'}$ if $\R\subseteq \B'\subseteq \B$. For any morphism $g\colon A\to B$ in $\B$, we denote by $\overline{g}$ the image of $g$ under the natural quotient functor $\B\to \oB$. For objects $A,B\in\B$, let $[R](A,B)$ be the subgroup of $\Hom_{\B}(A,B)$ consisting of morphisms which factor through an object in $\add R$.

\begin{defn}
For two objects $M,N\in\B$, denote by $\overline {[R]}(M,\Sigma N)$ the subset of $\Hom_{\B}(M,\Sigma N)$ such that $\alpha\in \overline {[R]}(M,\Sigma N)$ if we have $\alpha: M\xrightarrow{~h~} R_0 \xrightarrow{~j~} \Sigma N$ where $R_0\in \add R$ {\rm(}which means $\alpha \in [R](M,\Sigma N)${\rm )}, and the following commutative diagram:
$$\xymatrix@C=0.8cm@R0.6cm{
&M\ar[r]^h \ar@{.>}[d] &R_0 \ar[d]^j\\
N \ar[r] &I_N \ar[r]^i &\Sigma N \ar@{-->}[r] &
}
$$
where $I_N\in \mathcal I$.
\end{defn}

\begin{defn}\label{d1}
An object $X$ is called relative rigid (with respect to $\R$) if $\overline {[R]}(X,\Sigma X)={[R]}(X,\Sigma X)$.
\end{defn}

For convenience, a relative rigid object in this paper is also called $\R$-rigid.




\begin{rem}
Let $\B$ be a triangulated category with the shift functor $[1]$. Then any object $X$ is $\R$-rigid if and only if $X[1]$ is $R[1]$-rigid in the sense of \cite[Definition 2.2]{FGL}.
\end{rem}

\begin{lem}\label{summand}
Any $\R$-rigid object is closed under direct summands.
\end{lem}

\begin{proof}
Let $X\oplus X'$ be $\R$-rigid and $X\oplus X'\xrightarrow{\svech{h}{h'}} R_0$ be a left $\R$-approximation of $X\oplus X'$. Then $h$ is a left $\R$-approximation of $X$. For any morphism $\alpha: X\xrightarrow{~h~} R_0 \xrightarrow{~j~} \Sigma X$, we consider the following commutative diagram
$$\xymatrix{
X \ar[r]^-{\svecv{1}{0}} &X\oplus X' \ar[r]^-{\svech{h}{h'}} &R_0 \ar[d]^j\\
X \ar[r] \ar[d]_-{\svecv{1}{0}} &I_X \ar[r]^i \ar[d]^{i_1} &\Sigma X \ar[d]^a \ar@{-->}[r] &\\
X\oplus X' \ar[d]_-{\svech{0}{1}} \ar[r] &I \ar[d]^{i_2} \ar[r]^{i'\quad} &\Sigma X\oplus \Sigma X' \ar[d]^b \ar@{-->}[r] &\\
X \ar[r] &I_X \ar[r]^i &\Sigma X \ar@{-->}[r] &
}
$$
Since $X\oplus X'$ is $\R$-rigid, there is a morphism $X\oplus X' \xrightarrow{\svech{x}{x'}} I$ such that $i'\svech{x}{x'}=aj\svech{h}{h'}$. There is a morphism $y:\Sigma X\to I_X$ such that $1-ba=iy$. Now we have $ii_2x=bi'x=bajh=(1-iy)jh$, hence $jh=i(yjh+i_2x)$, which implies $X$ is an $\R$-rigid object.
\end{proof}

We have the following useful lemma.

\begin{lem}\label{rigid}
$X$ is $\R$-rigid if and only if there exists the following commutative diagram
$$\xymatrix{
\Omega X \ar[r]^s \ar[d]^q &\Omega R_1 \ar@{=}[r] \ar[d]^f  &\Omega R_1 \ar[d]^{p_1}\\
P_X \ar[r] \ar[d] &\Omega R_2 \ar[r] \ar[d]^g &P \ar[r] \ar[d] &R_2 \ar@{=}[d] \ar@{-->}[r] &\\
X \ar@{=}[r] \ar@{-->}[d] &X \ar[r]_h \ar@{-->}[d] &R_1 \ar[r] \ar@{-->}[d] &R_2 \ar@{-->}[r] &\\
& & &
}
$$
where $R_1,R_2\in \R$ and $P,P_X\in \mathcal P$, $\Hom_{\uB}(\underline f,X)$ is surjective.
\end{lem}

\begin{proof}
X admits an $\EE$-triangle $\xymatrix@C=0.8cm@R0.6cm{X \ar[r] &I_X \ar[r]^i &\Sigma X \ar@{-->}[r] &}$. Let $j:R_1\to \Sigma X$ (resp. $x:\Omega R_1\to X$) be any morphism, we can the following commutative diagram
$$\xymatrix{
\Omega X \ar[r]^s \ar[d]^q &\Omega R_1 \ar@{=}[r] \ar[d]^f  &\Omega R_1 \ar[d]^{p_1} \ar[r]^x &X \ar[d]\\
P_X \ar[r]^{a} \ar[d] &\Omega R_2 \ar[r]^b \ar[d]^g &P \ar[r] \ar[d] &I_X \ar[d]^i \\
X \ar@{=}[r] \ar@{-->}[d] &X \ar[r]_h \ar@{-->}[d] &R_1 \ar[r]_j \ar@{-->}[d] &\Sigma X \ar@{-->}[d]\\
& & & &
}
$$

If $\Hom_{\uB}(\underline f,X)$ is surjective, we show that $jh$ factors through $i$, which by definition means $X$ is $\R$-rigid.\\

Now let $j:R_1\to \Sigma X$ be any morphism. There is a morphism $x_1:\Omega R_2\to X$ such that $\underline {x}=\underline {x_1f}$, then $x-x_1f$ factors through $\mathcal P$, hence we have $x-x_1f:\Omega R_1\xrightarrow{p_1} P \xrightarrow{p_2} X$. Thus $$xs=x_1fs+p_2p_1s=x_1aq+p_2baq=(x_1+p_2b)aq,$$ this implies $jh$ factors through $i$.

Now assume that $X$ is $\R$-rigid, we show that $\Hom_{\B}(f,X)$ is surjective, hence $\Hom_{\uB}(\underline f,X)$ is surjective.
Let $x:\Omega R_1\to X$ be any morphism. Since $X$ is $\R$-rigid, $jh$ factors through $i$, which implies $xs$ factors through $q$. Hence $x$ factors through $f$.
\end{proof}

\begin{lem}\label{sum}
Let $X_0$ be an $\R$-rigid and $R_0\in \R$. Then $X_0\oplus R_0$ is $\R$-rigid if and only if $\EE(R_0,X_0)=0$.
\end{lem}

\begin{proof}
Let $y:R_0\to \Sigma X_0$ be any morphism. Then we have the following commutative diagram:
$$\xymatrix@C=1.5cm{
&X_0\oplus R_0 \ar[r]^-{\left(\begin{smallmatrix}
h_0&0\\
0&1
\end{smallmatrix}\right)} &R_1\oplus R_0 \ar[d]^-{\svech{0}{1}}\\
\Omega R_0 \ar[d]^x \ar[r] &P_0 \ar[d] \ar[r] &R_0 \ar[d]^y \ar@{-->}[r] &\\
X_0 \ar[r]^{j_0} \ar[d]^{\svecv{1}{0}} &I \ar[r]^{i_0} \ar[d]^{\svecv{1}{0}} &\Sigma X_0 \ar[d]^{\svecv{1}{0}} \ar@{-->}[r] &\\
X_0\oplus R_0 \ar[r]_-{\left(\begin{smallmatrix}
j_0&0\\
0&j_R
\end{smallmatrix}\right)} &I\oplus I_R \ar[r]_-{\left(\begin{smallmatrix}
i_0&0\\
0&i_R
\end{smallmatrix}\right)=\alpha} &\Sigma X_0\oplus \Sigma R_0 \ar@{-->}[r] &
}
$$
where $h_0$ is a left $\R$-approximation of $X_0$ and $P_0\in \mathcal P$, $I, I_R\in \mathcal I$.
\medskip

If $X_0\oplus R_0$ is $\R$-rigid, there is a morphism $X_0\oplus R_0 \xrightarrow{\left(\begin{smallmatrix}
a&b\\
c&d
\end{smallmatrix}\right)} I\oplus I_R$ such that $$\left(\begin{smallmatrix}
i_0&0\\
0&i_R
\end{smallmatrix}\right)\left(\begin{smallmatrix}
a&b\\
c&d
\end{smallmatrix}\right)=\svecv{1}{0}y\svech{0}{1}\left(\begin{smallmatrix}
h_0&0\\
0&1
\end{smallmatrix}\right).$$
It follows that $i_0b=y$ and then $\EE(R_0,X_0)=0$.

Let $\beta=\left(\begin{smallmatrix}
a'&b'\\
c'&d'
\end{smallmatrix}\right):R_1\oplus R_0\to \Sigma X_0\oplus \Sigma R_0$ be any morphism. We have $b':R_1\to \Sigma R_0$ and $d':R_0\to \Sigma R_0$ factor through $i_R$ since $\R$ is rigid, $a':R_1\to \Sigma X_0$ factors through $i_0$ since $X_0$ is $\R$-rigid. If $\EE(R_0,X)=0$, we have $c':R_0\to \Sigma X_0$ factors through $i_0$, this implies $\beta$ factors through $\alpha$. Hence by definition $X_0\oplus R_0$ is $\R$-rigid.
\end{proof}

\subsection{Relative maximal rigid objects}
Let $\h=\CoCone(\R,\R)$. Then according to \cite{LN}, $\Hom_{\uB}(\Omega R,-)$ induces an equivalence $G:\overline{\h}\xrightarrow{~\simeq~} \mod\End_{\uB}(\Omega R)=:\Gamma$. In fact we have the following commutative diagram:
$$\xymatrix@C=0.8cm@R0.5cm{
\B \ar[dr]_H \ar[rr]^-{\Hom_{\uB}(\Omega R,-)} &&\Gamma\\
&\overline \h \ar[ur]_G^{\simeq}
}
$$
where $H$ is the cohomological functor defined in \cite{LN}. $H$ has also the following properties:
\begin{itemize}
\item[(1)] $H(X)=X$ if $X\in \h$; $H(f)=0$ if and only if $f$ factors through  $\R^{\bot_1}=\{X\in\B~|~\Ext^1(\R,X)=0\}$.
\item[(2)] By applying $H$ to an $\EE$-triangle $\xymatrix{X\ar[r]^f &B \ar[r]^g &C \ar@{-->}[r] &}$, we get an exact sequence $$H(X) \xrightarrow{H(f)} H(Y) \xrightarrow{H(g)} H(Z).$$
\end{itemize}

\begin{defn}\label{d2}
An object $X\in \h$ is called relative maximal rigid with respect to $\h$ if for any non-projective indecomposable object $Z\in \h$, $X\oplus Z$ is $\R$-rigid implies $Z\in \add X$. For convenience, $X$ is also called a maximal $\R$-rigid object.
\end{defn}

\begin{lem}\label{eq}
For any $X,Y\in \h$, we have $\Hom_{\Gamma}(\Hom_{\uB}(\Omega R,X),\Hom_{\uB}(\Omega R,Y))\simeq \Hom_{\oB}(X,Y)$.
\end{lem}

\subsection{$\tau$-tilting theory}Let $\Lambda$ be a finite dimensional basic algebra. We denote by $\mathsf{proj}\Lambda$ the subcategory consisting of
finite dimensional left projective $\Lambda$-modules. For each $\Lambda$-module $X$, we denote the number of
non-isomorphic indecomposable direct summands of $X$ by $|X|$. Let $\tau$ be the Auslander-Reiten translation
in $\mod\Lambda$, we recall in the following some basic concepts in $\tau$-tilting theory.
\begin{defn}\cite[Definition 0.3]{AIR}
Let $(X, P)$ be a pair with $X\in \mod\Lambda$ and $P\in\mathsf{proj}\Lambda$.
\begin{itemize}
\item We say that $(X, P)$ is {\rm basic} if $X$ and $P$ are basic.
\item $X$ is called $\tau${\rm -rigid} if {\rm Hom}$_{\Lambda}(X,\tau X)=0$.
\item $(X, P)$ is called a $\tau${\rm -rigid pair} if $X$ is $\tau$-rigid and {\rm Hom}$_{\Lambda}(P, X)=0$.
\item $X$ is called $\tau${\rm -tilting} if $X$ is $\tau$-rigid and $|X|=|\Lambda|$.
\item A $\tau$-rigid pair $(X, P)$ is said to be a {\rm support $\tau$-tilting} pair if $|X|+|P|=|\Lambda|$. In this case, $X$ is also called a {\rm support $\tau$-tilting module}.
\end{itemize}
\end{defn}

We need the following lemma, which has been proved in \cite[Proposition 2.4]{AIR}.

\begin{lem}\label{AIR}
For any object $M\in \mod \Lambda$, denote by $P^M_1\xrightarrow{~f~} P^M_0\to M\to 0$  a minimal projective presentation of $M$. Then $M$ is $\tau$-rigid if and only if $\Hom_{\Lambda}(f,M)$ is surjective.
\end{lem}

Now we state and prove our first theorem. This result generalizes their work by Adachi-Iyama-Reiten \cite[Theorem 4.1]{AIR}, Yang-Zhu \cite[Theorem 3.6]{YZ} and Fu-Geng-Liu \cite[Theorem 2.5]{FGL}. 

\begin{thm}\label{main}
\begin{itemize}
\item[(a)] Let $X\in \h$. Then $X$ is $\R$-rigid if and only if $\Hom_{\uB}(\Omega R,X)$ is $\tau$-rigid.
\item[(b)] $\Hom_{\uB}(\Omega R,-)$ yields a bijection between the set of isomorphism classes of basic $\R$-rigid objects in $\h$ which have no direct summands in $\mathcal P$ and the set of isomorphism classes of basic $\tau$-rigid pairs of $\Gamma$-modules.
\item[(c)] $\Hom_{\uB}(\Omega R,-)$ yields a bijection between the set of isomorphism classes of basic maximal $\R$-rigid objects in $\h$ which have no direct summands in $\mathcal P$ and the set of isomorphism classes of basic support $\tau$-tilting pairs of $\Gamma$-modules.
\end{itemize}
\end{thm}

\begin{proof}
(a) Let $\Hom_{\uB}(\Omega R,\Omega R_1) \xrightarrow{\alpha} \Hom_{\uB}(\Omega R,\Omega R_0) \xrightarrow{\beta} \Hom_{\uB}(\Omega R,X)\to 0$ be a minimal projective resolution. By Yoneda's Lemma, we have $\alpha=\Hom_{\uB}(\Omega R,\underline f)$. Thus we have the following commutative diagram by Lemma \ref{eq}.
$$\xymatrix{
\Hom_{\uB}(\Omega R_0,X) \ar[d]_-{\Hom_{\uB}(\underline f,X)} \ar[rr]^-{\simeq} &&\Hom_{\Gamma}(\Hom_{\uB}(\Omega R,\Omega R_0),\Hom_{\uB}(\Omega R,X))\ar[d]^-{\Hom_{\Gamma}(\Hom_{\uB}(\Omega R, \underline f),\Hom_{\uB}(\Omega R,X))}\\
\Hom_{\uB}(\Omega R_1,X) \ar[rr]^-{\simeq} &&\Hom_{\Gamma}(\Hom_{\uB}(\Omega R, \Omega R_1),\Hom_{\uB}(\Omega R,X))
}
$$
By Lemma \ref{AIR} and \ref{rigid}, $X$ is $\R$-rigid if and only if $\Hom_{\uB}(\Omega R,X)$ is $\tau$-rigid.

(b) Let $X$ be a basic $\R$-rigid object, $X=X_0\oplus R_0$ where $X_0$ has no direct summands in $\R$ and $R_0\in \R$. Let $F(X)=(\Hom_{\uB}(\Omega R,X_0),\Hom_{\uB}(\Omega R,R_0))$. By Lemma \ref{summand}, $X_0$ is $\R$-rigid, then by (a), $\Hom_{\uB}(\Omega R,X_0)$ is $\tau$-rigid. Since $X$ is $\R$-rigid, by Lemma \ref{sum}, we have $\Hom_{\uB}(\Omega R_0,X)\simeq \EE(R_0,X)=0$, and by Lemma \ref{eq}, $$\Hom_{\Gamma}(\Hom_{\uB}(\Omega R, \Omega R_0),\Hom_{\uB}(\Omega R,X))=0,$$ hence $(\Hom_{\uB}(\Omega R,X),\Hom_{\uB}(\Omega R, \Omega R_0))$ is a basic $\tau$-rigid pair. On the other hand, if we have a basic $\tau$-rigid pair $(M,P)$ of $\Gamma$, then $M\simeq \Hom_{\uB}(\Omega R,X_0)$ where $X_0$ has no direct summands in $\R$ and $P\simeq \Hom_{\uB}(\Omega R, \Omega R_0)$. By (a) $X_0$ is $\R$-rigid. Since $\Hom_{\uB}(\Omega R_0,X_0)=0$, by Lemma \ref{sum} $X_0\oplus R_0$ is a $\R$-rigid object.

(c) Let $X=X_0\oplus R_X$ where $X_0$ does not have any direct summand of $\R$. If $X$ is maximal rigid and $F(X)$ is not basic support $\tau$-tilting, then one of the following conditions must be satisfied:
\begin{itemize}
\item[(1)] There exists an indecomposable object $R'\in \add R$ which is not a direct summand of $R_X$ such that $(\Hom_{\uB}(\Omega R,X),\Hom_{\uB}(\Omega R,\Omega R_X\oplus \Omega R'))$ is a basic $\tau$-rigid pair.
\item[(2)] There exists an indecomposable non-projective object $X'$ which is not a direct summand of $X_0$ or $R$ such that $(\Hom_{\uB}(\Omega R,X\oplus X'),\Hom_{\uB}(\Omega R,\Omega R_X))$ is a basic $\tau$-rigid pair.
\end{itemize}
For condition (1), it is enough to show that $X_0\oplus R_X\oplus R'$ is $\R$-rigid.

Since $(\Hom_{\uB}(\Omega R,X_0),\Hom_{\uB}(\Omega R,\Omega R_X\oplus \Omega R'))$ is a $\tau$-rigid pair, we have $$\EE(R_X\oplus R',X_0)\simeq \Hom_{\uB}(\Omega R_X\oplus \Omega R',X_0)=0,$$ hence $\EE(R',X_0)=0$ and by Lemma \ref{sum}, $X\oplus R'$ is $\R$-rigid, a contradiction to the fact that $X$ is maximal.

By the same method we can show that condition (2) can not hold. Hence $F(X)$ is basic support $\tau$-tilting.

If $(M,P)$ is basic support $\tau$-tilting and $M\simeq \Hom_{\uB}(\Omega R,X_0)$ where $X_0$ has no direct summands in $\R$, $P\simeq \Hom_{\uB}(\Omega R, \Omega R_0)$. By (b) $X_0\oplus R_X=:X$ is $\R$-rigid, assume that $Z\in \h$ is a non-projective indecomposable object, we show that $X\oplus Z$ is $\R$-rigid implies $Z\in \add X$.

If $Z\in \add R$, then by (b) we have a basic $\tau$-rigid pair $(\Hom_{\uB}(\Omega R,X_0),\Hom_{\uB}(\Omega R,\Omega R_X\oplus \Omega Z))$. Since $(\Hom_{\uB}(\Omega R,X_0),\Hom_{\uB}(\Omega R,\Omega R_X)$ is already a basic support $\tau$-tilting pair, we have $Z\in \add R_X$. By the same method we can show the case when $Z\notin \add R$, $Z\in \add X_0$.
\end{proof}

For any object $X$, denote by $\underline {[X]}(R,R)$ the image of $[X](R,R)$ in $\underline \B$. Now we give an equivalent characterization of tilting modules.

\begin{thm}\label{main1}
For an object $X\in \h$ which does not have direct summands in $\R$, we have $$\pd_{\Gamma}\Hom_{\uB}(\Omega R, X)\leq 1~ \textrm{if and only if }~\underline {[X]}(R,R)=0.$$ Moreover, if $X$ is basic, then $\Hom_{\uB}(\Omega R, X)$ is a tilting module of $\Gamma$ if and only if $X$ is maximal $\R$-rigid with respect to $\h$ and $\underline {[X]}(R,R)=0$.
\end{thm}

\begin{proof}
$X$ admits following commutative diagram
$$\xymatrix{
\Omega X \ar[r]^s \ar[d]^q &\Omega R_1 \ar@{=}[r] \ar[d]^f  &\Omega R_1 \ar[d]^{p_1}\\
P_X \ar[r] \ar[d] &\Omega R_2 \ar[r] \ar[d]^g &P \ar[r] \ar[d] &R_2 \ar@{=}[d] \ar@{-->}[r] &\\
X \ar@{=}[r] \ar@{-->}[d] &X \ar[r]_h \ar@{-->}[d] &R_1 \ar[r] \ar@{-->}[d] &R_2 \ar@{-->}[r] &\\
& & &
}
$$
where $R_1,R_2\in \R$ and $P,P_X\in \mathcal P$.

By applying $\Hom_{\uB}(\Omega R, -)$, we get an exact sequence:
$$(\Omega R, \Omega X)\xrightarrow{\Hom_{\uB}(\Omega R, \underline s)} (\Omega R, \Omega R_1) \xrightarrow{\Hom_{\uB}(\Omega R, \underline f)} (\Omega R, \Omega R_2) \xrightarrow{\Hom_{\uB}(\Omega R, \underline g)}(\Omega R, X)\to 0$$
where we omitted $\Hom_{\uB}$ because of lack of space.

Let $a\in \Hom_{\B}(\Omega R, \Omega X)$ be any morphism. Thus we have the following commutative diagram
$$\xymatrix{
\Omega R \ar[r] \ar[d]_a &P_R \ar[r] \ar[d] &R \ar[d]^b \ar@{-->}[r] &\\
\Omega X \ar[r] \ar[d]_s &P_X \ar[r] \ar[d] &X \ar[d]^h \ar@{-->}[r] &\\
\Omega R_1 \ar[r] & P \ar[r] &R_1 \ar@{-->}[r] &\\
}
$$
If $\underline {[X]}(R,R)=0$, that is to say, $hb$ factors through $P$ which implies $sa$ factors through $P_R$, hence $\Hom_{\uB}(\Omega R, \underline s)=0$ and $\pd_{\Gamma}\Hom_{\uB}(\Omega R, X)\leq 1$.\\[2mm]
If $\pd_{\Gamma}\Hom_{\uB}(\Omega R, X)\leq 1$, then we have the following exact sequence
$$0\to \Hom_{\uB}(\Omega R, \Omega R_1) \xrightarrow{~\alpha~} \Hom_{\uB}(\Omega R, \Omega R_2) \xrightarrow{~\beta~}\Hom_{\uB}(\Omega R, X)\to 0$$
where $\alpha=\Hom_{\uB}(\Omega R, \underline f)$ by Yoneda's Lemma. Thus we have the following commutative diagram
$$\xymatrix{
\Omega R_1 \ar[r]^{p_1} \ar[d]_f &P \ar[r] \ar[d] & R_1 \ar@{=}[d] \ar@{-->}[r] &\\
\Omega R_2 \ar[r] &X' \ar[r] &R_1 \ar@{-->}[r] &
}
$$
where $X'\in \h$ and $\Hom_{\uB}(\Omega R, X)\simeq \Hom_{\uB}(\Omega R, X')$. But this also means $X=H(X)\simeq H(X')=X'$ in $\overline \h$. Since $X$ does not have direct summands in $\R$, we have $X'=X\oplus R'\oplus P'$ where $R'\in \add R$ and $P'\in \mathcal P$. To show $\underline {[X]}(R,R)=0$, it is enough to show $\underline {[X']}(R,R)=0$. We have the following commutative diagram
$$\xymatrix{
\Omega X' \ar[r]^{s'} \ar[d] &\Omega R_1 \ar@{=}[r] \ar[d]_-{\svecv{f}{p_1}}  &\Omega R_1\ar[d]^-{\svecv{p_2f}{p_1}} \\
P' \ar[d] \ar[r] &\Omega R_2\oplus P \ar[d] \ar[r]^-{\left(\begin{smallmatrix}
p_2&0\\
0&1
\end{smallmatrix}\right)} &P_2\oplus P \ar[r]_-{\svech{q_2}{0}} \ar[d] &R_2 \ar@{=}[d] \ar@{-->}[r]&\\
X' \ar@{-->}[d] \ar@{=}[r] &X' \ar[r]_{h'} \ar@{-->}[d] &R_1' \ar[r] \ar@{-->}[d] &R_2 \ar@{-->}[r] &\\
&& &
}
$$
where $P,P_2,P'\in \mathcal P$, $R_1',R_2\in \R$ and $\Hom_{\uB}(\Omega R, \underline s')=0$. Let $b':R\to X'$ be any morphism. Then we have the following commutative diagram
$$\xymatrix{
\Omega R \ar[r] \ar[d]_a &P_R \ar[r] \ar[d] &R \ar[d]^{b'} \ar@{-->}[r] &\\
\Omega X' \ar[r] \ar[d]_{s'} &P' \ar[r] \ar[d] &X' \ar[d]^{h'} \ar@{-->}[r] &\\
\Omega R_1 \ar[r] & P_2\oplus P \ar[r] &R_1' \ar@{-->}[r] &\\
}
$$
Since $s'a$ factors through $\mathcal P$, we have $h'b'$ factors through $\mathcal P$. Since $h'$ is a left $\R$-approximation of $X'$, we have $\underline {[X']}(R,R)=0$ implies $\underline {[X]}(R,R)=0$.
\medskip

If $X$ is maximal $\R$-rigid, by Theorem \ref{main}, $\Hom_{\uB}(\Omega R, X)$ is support $\tau$-tilting. $X$ does not have direct summand in $\R$ implies that $|\Hom_{\uB}(\Omega R, X)|=|X|=|\Gamma|=|R|$. This condition $\underline {[X]}(R,R)$ means $\pd_{\Gamma}\Hom_{\uB}(\Omega R, X)\leq 1$. Combine all of these, we get that $\Hom_{\uB}(\Omega R, X)$ is a tilting module.
\medskip

If $\Hom_{\uB}(\Omega R, X)$ is a tilting module, since it is also support $\tau$-tilting, by Theorem \ref{main} we know that $X$ is maximal $\R$-rigid. By definition of tilting module, we have $\pd_{\Gamma}\Hom_{\uB}(\Omega R, X)\leq 1$, hence $\underline {[X]}(R,R)=0$.
\end{proof}

Now we study the relationship between $\R$-rigid, rigid and $d$-rigid.

\begin{prop}\label{2rigid}
If $R$ is $2$-rigid, then any $\R$-rigid object $X\in \h$ is rigid.
\end{prop}

\begin{proof}
Assume that $X$ is any $\R$-rigid object, it admits an $\EE$-triangle $\xymatrix{X \ar[r] ^h &R_1 \ar[r]^j &R_2 \ar@{-->}[r] &}$ where $R_1,R_2\in \add R$. If $R$ is $2$-rigid, we get a short exact sequence $0=\EE(R,R_2)\to \EE^2(R,X) \to \EE^2(R,R_1)=0$ which implies $\EE^2(R,X)=0$. Then we have a short exact sequence  $$\EE(R_1,X) \xrightarrow{\EE(f,X)} \EE(X,X)\to \EE^2(R_2,X)=0.$$
This implies that any $\EE$-triangle $\xymatrix{X\ar[r]^x &Y\ar[r] &X\ar@{-->}[r] &}$ admits the following commutative diagram
$$\xymatrix{
X \ar[r]^x \ar@{=}[d] &Y \ar[r] \ar[d] &X \ar@{-->}[r] \ar[d]^h &\\
X \ar[r] \ar@{=}[d] &Z \ar[r] \ar[d] &R_1 \ar@{-->}[r] \ar[d]^a &\\
X \ar[r] &I_X \ar[r]_i &\Sigma X \ar@{-->}[r] &
}
$$
Since $X$ is $\R$-rigid, we have that $ah$ factors through $i$, hence $1_X$ factors through $x$. It follows that the $\EE$-triangle $\xymatrix{X\ar[r]^x &Y\ar[r] &X\ar@{-->}[r] &}$ splits. Thus $\EE(X,X)=0$ and then $X$ is rigid.
\end{proof}

\begin{thm}
Let $X\in \h$. If $R$ is $d$-rigid, $d\geq 3$, then the following statements are equivalent:
\begin{itemize}
\item[(a)] $X$ is $\R$-rigid;
\item[(b)] $X$ is rigid;
\item[(c)] $X$ is $(d-1)$-rigid.
\end{itemize}
\end{thm}

\begin{proof}
By Proposition \ref{2rigid}, (a) and (b) are equivalent. (c) implies (a) is trivial. We show that if $\R$ is $d$-rigid, then $\EE^i(X,X)=0$, $i=2,3,\cdots,d-1$, hence (b) implies (c).

By the same method as in Proposition \ref{2rigid}, we get $0=\EE^{i-1}(R,R_2)\to \EE^i(R,X) \to \EE^i(R,R_1)=0$ which implies $\EE^i(R,X)=0$, $i=2,3,\cdots,d$. Thus we have short exact sequences $$0=\EE^i(R_1,X) \to \EE^i(X,X)\to \EE^{i+1}(R_2,X)=0,$$ hence $\EE^i(X,X)=0$, $i=2,3,\cdots,d-1$.
\end{proof}

\section{A partial order on relative maximal rigid objects}


Let $M\in \h$, we say $X\in M*[R]$ if $X$ admits the following commutative diagram:

$$\xymatrix{
M_X \ar[r]^i \ar[d]_m &I_M \ar[d]^{i'} \ar[r] &\Sigma M_X \ar@{=}[d] \ar@{-->}[r] & \ar@{}[d]^{(\lozenge)}\\
X \ar[r]_{\eta_X} \ar[r] &C_X \ar[r] &\Sigma M_X \ar@{-->}[r] &
}
$$
where $M_X\in \add M$, $I_M\in \mathcal I$ and $\eta_X$ factors through $\R^{\bot_1}$.

The following remark is useful, the proof is left to the readers.

\begin{rem}
The subcategory $M*[R]$ is closed under direct sums and direct summands.
\end{rem}

\begin{lem}\label{R}
If $X\in \h$ and $X\in M*[R]$, then in the diagram $(\lozenge)$, $\eta_X$ factors through $\R$.
\end{lem}

\begin{proof}
Since $X\in \h$ and $X\in M*[R]$, we get an epimorphism $M_X \xrightarrow{\overline m} X\to 0$ in $\overline \h$. According to \cite[Corollary 2.26]{LN}, $m$ admits the following commutative diagram
$$\xymatrix{
M_X \ar[r]^r \ar[d]_m &R_1 \ar[r] \ar[d] &R_2 \ar@{=}[d] \ar@{-->}[r] & \\
X \ar[r] &R_0 \ar[r]  &R_2 \ar@{-->}[r] &
}
$$
where $R_0,R_1,R_2\in \R$. There exists a morphism $r':R_1\to I_M$ such that $r'r=i$. Then we have the following commutative diagram
$$\xymatrix{
M_X \ar[r]^r \ar[d]_m &R_1 \ar[d]  \ar[rdd]^{i'r'} \\
X \ar[r] \ar[rrd]_{\eta_X} &R_0 \ar@{.>}[rd] \\
&&C_X
}
$$
which shows $\eta_X$ factors through $\R$.
\end{proof}

\begin{defn}\label{order}
Let $M,N$ be maximal $\R$-rigid. We denote $M \geq N$ if $M*[R]\supseteq N*[R]$.
\end{defn}

\begin{thm}\label{partial}
The relation $\geq$ we defined is a partial order.
\end{thm}

\begin{proof}
We only need to show that if $M\geq N$ and $N\geq M$, then $M=N$. To show this, we only need to prove $M\oplus N$ is $\R$-rigid. Since $M$ and $N$ are maximal, if $M\oplus N$ is $\R$-rigid, we have $M=N\oplus M=N$.\\
Now assume $M*[R]\subseteq N*[R]$ and $N*[R]\subseteq M*[R]$. Then $N\in M*[R]$. By definition we have the following commutative diagram
$$\xymatrix{
M_N \ar[r] \ar[d] &I_M \ar[d] \ar[r] &\Sigma M_N \ar@{=}[d] \ar@{-->}[r] &\\
N \ar[r]_{\eta_N}  &C_N \ar[r] &\Sigma M_N \ar@{-->}[r] &
}
$$
where $M_N\in \add M$, $I_M\in \mathcal I$ and $\eta_N$ factors through $\R^{\bot_1}$. We check that in the following commutative diagram
$$\xymatrix{
\Omega M \ar[r]^{p_M} \ar[d]_{f_1} &P_M \ar[d] \ar[r] &M \ar[d]^{g_1} \ar@{-->}[r] &\\
\Omega R \ar[r]^{p_R} \ar[d]_{f_2} &P_R \ar[r] \ar[d] &R \ar[d]^{g_2} \ar@{-->}[r] &\\
N \ar[r] &I_N \ar[r]_{i_N} &\Sigma N \ar@{-->}[r] &
}
$$
where $P_R,P_M\in \mathcal P$, $I\in \mathcal I$, the morphism $g_2g_1$ factors through $i_N$. It is enough to show $f_2f_1$ factors through $p_M$. In fact, since $\eta_N$ factors through $\R^{\bot_1}$, there exists a morphism $p_1:P_R\to C_N$ such that $\eta_Nf_2=p_1p_R$. Then we have the following diagram such that all the squares are commutative:
$$\xymatrix{
\Omega R \ar[r]^{p_R} \ar[rdd]_{f_2} \ar@{.>}[rd]^a &P_R \ar@{.>}[rdd] \ar[r] \ar@{.>}[rd] &R \ar@{.>}[rd]^r \ar@{.>}[rdd] \ar@{-->}[r] &\\
&M_N \ar[d]^b \ar[r]  &I_M \ar[d] \ar[r] &\Sigma M_N \ar@{=}[d] \ar@{-->}[r] &\\
&N \ar[r]_{\eta_N}  &C_N \ar[r]  &\Sigma M_N  \ar@{-->}[r] &
}$$
We have that $ba$ factors through $p_R$. There is a morphism $p_2:P_R\to N$ such that $f_2-ba=p_2p_R$. Hence $f_2f_1=baf_1+p_2p_Rf_1$ factors though $p_M$.

The dual of this statement can be shown under this condition $M\in N*[R]$.
\end{proof}

\begin{lem}\label{Fac}
Let $X\in \B$ and $M\in \h$. Then $X\in M*[R]$ if and only if $H(X)\in \Fac M$ in $\overline \h$.
\end{lem}

\begin{proof}
If $X\in M*[R]$, then we have a commutative diagram
$$\xymatrix{
M_X \ar[r] \ar[d] &I_M \ar[d] \ar[r] &\Sigma M_X \ar@{=}[d] \ar@{-->}[r] &\\
X \ar[r]_{\eta_X} \ar[r] &C_X \ar[r] &\Sigma M_X \ar@{-->}[r] &
}
$$
where $\eta_X$ factors through $\R^{\bot_1}$. Thus we get an epimorphism $M_X\to H(X)\to 0$ in $\overline \h$. It folllows that $H(X)\in \Fac M$.

If we have an epimorphism $\overline f:M^n\to H(X)\to 0$, since $X$ admits an $\EE$-triangle $$\xymatrix{R_X \ar[r]^{r\quad} &H(X) \ar[r] &X\ar@{-->}[r] &}$$ where $R_X\in \R$ and $H(r)$ is an isomorphism in $\overline \h$, then we have a commutative diagram
$$\xymatrix{
M^n \ar[r] \ar[d]_{rf} &I \ar[d] \ar[r] &\Sigma M^n \ar@{=}[d] \ar@{-->}[r] &\\
X \ar[r]_{\eta_X} \ar[r] &C_X \ar[r] &\Sigma M^n \ar@{-->}[r] &
}
$$
where $I\in \mathcal I$. We get an exact sequence $M^n \xrightarrow{H(r)\overline f} H(X) \xrightarrow{ H(\eta_X)} C_X$ in $\oB$. Since $H(r)\overline f$ is epic, we have $H( \eta_X)=0$, which implies $\eta_X$ factors through $\R^{\bot_1}$.
\end{proof}

According to this lemma, we have the following proposition.

\begin{prop}\label{mn}
Let $M,N$ be two objects in $\h$. Then $M*[R]\subseteq N*[R]$ if and only if $\Fac M\subseteq \Fac N$ in $\overline \h$.
\end{prop}

\begin{proof}
It is enough to show if $M\in N*[R]$, then $M*[R]\subseteq N*[R]$.\\
Assume $X\in M*[R]$, by definition it admits a commutative diagram
$$\xymatrix{
M_X \ar[r] \ar[d]_m &I_M \ar[d] \ar[r] &\Sigma M_X \ar@{=}[d] \ar@{-->}[r] &\\
X \ar[r]_{\eta_X} \ar[r] &C_X \ar[r] &\Sigma M_X \ar@{-->}[r] &
}
$$
where $M_X\in \add M$, $I_M\in \mathcal I$ and $\eta_X$ factors through $\R^{\bot_1}$. 
Since $M\in N*[R]$ and $N*[R]$ is closed under direct sums and direct summands, then $M_X$ also admits a commutative diagram
$$\xymatrix{
N' \ar[r] \ar[d]_n &I \ar[d] \ar[r] &\Sigma N' \ar@{=}[d] \ar@{-->}[r] &\\
M_X \ar[r]_{\eta} \ar[r] &C \ar[r] &\Sigma N' \ar@{-->}[r] &
}
$$
where $N'\in \add N$, $I\in \mathcal I$ and $\eta$ factors through $\R^{\bot_1}$. Now consider the following commutative diagram
$$\xymatrix{
N' \ar[r] \ar[d]_{mn} &I \ar[d] \ar[r] &\Sigma N' \ar@{=}[d] \ar@{-->}[r] &\\
X \ar[r]_{\eta_X'} \ar[r] &C_X' \ar[r] &\Sigma N' \ar@{-->}[r] &
}
$$
since $H (m)$ and $\overline n$ are surjective in $\overline \h$, then $H(m)\overline n$ is also surjective. But we have an exact sequence $$N'\xrightarrow{H(m)\overline n} H(X)\xrightarrow{H(\eta_X')} C_X',$$ thus $H(\eta_X')=0$ and $\eta_X'$ factors through $\R^{\bot_1}$. This means $X\in N*[R]$.
\end{proof}

\section{Mutation of relative maximal rigid objects}

Let $M=U\oplus X\in \h$ be a basic maximal $\R$-rigid object where $X$ is a non-projective indecomposable object. In this case, $U$ is called \emph{relative almost maximal rigid object} (also is called \emph{almost maximal $\R$-rigid objects}). Thus by Theorem \ref{main} and \cite[Theorem 2.17]{AIR}, there exists another indecomposable object $Y\in \h$ such that $U\oplus Y=N$ is also maximal $\R$-rigid. By \cite[Definition-Proposition 2.26]{AIR} and Proposition \ref{mn}, we have either $M>N$ or $N>M$.

\begin{defn}
Let $M$, $N$ be two basic maximal $\R$-rigid objects in $\h$. We say $N$ is a left mutation of $M$ and $M$ is a right mutation of $N$ if the following condition are satisfied:
\begin{itemize}
\item[(1)] $M=U\oplus X$ and $N=U\oplus Y$, where $X,Y$ are non-projective indecomposable objects.
\item[(2)] $M>N$ by Definition \ref{order}.
\end{itemize}
\end{defn}

From now onwards, we assume $\mathcal P\subseteq \add U$, then we can take a minimal right $\add U$-approximation $f:U_1\to X$ which is also a deflation. We have the following lemma.

\begin{lem}
Let $M=U\oplus X$ be a basic maximal $\R$-rigid where $X$ is indecomposable and all the indecomposable projective objects are isomorphic to direct summands of $U$. If $X\in U*[R]$, then  we can get the following $\EE$-triangle
$$\xymatrix{
Y \ar[r]^g &U_1 \ar[r]^f &X \ar@{-->}[r] &
}
$$
where $f$ is a left minimal. Moreover, in the following commutative diagram
$$\xymatrix{
Y \ar[r]^g \ar@{=}[d] &U_1 \ar[r]^f \ar[d]^{u_0} &X\ar[d]^h \ar@{-->}[r] & \ar@{}[d]^{(\star)}\\
Y \ar[r] &I \ar[r]_i \ar[d]^{i'} &\Sigma Y \ar@{-->}[r] \ar[d]^{h'} &\\
&\Sigma U_1 \ar@{-->}[d] \ar@{=}[r] &\Sigma U_1\ar@{-->}[d]\\
& &&&}
$$
we have that $h$ factors through $\R$. Moreover, $Y\in \h$.
\end{lem}

\begin{proof}
Since $X\in U*[R]$, we have the following commutative diagram
$$\xymatrix{
&&M_X \ar[r]^{m_1} \ar[d]^m &I_M \ar[d] \ar[r] &\Sigma M_X \ar@{=}[d] \ar@{-->}[r] &\\
Y \ar[r] \ar@{=}[d] &U_1 \ar[r]^f \ar[d] &X\ar[d]^h \ar[r]_{\eta_X} \ar[r] &C_X \ar[r] &\Sigma M_X \ar@{-->}[r] &\\
Y \ar[r] &I \ar[r]_i &\Sigma Y \ar@{-->}[r] &
}
$$
where $\eta_X$ factors through $\R$ by Lemma \ref{R}.
Since $f$ is a right $U$-approximation of $X$, then $m$ factors through $f$, hence it also factors through $f$. It follows that $hm$ factors through $i$, thus $hm$ factors through $m_1$. Therefore $h$ factors through $\eta_X$, which means it factors through $\R$.

Now we get $\overline f: U_1\to X$ is surjective, then by \cite[Corollary 2.26]{LN}, $f$ admits the following commutative diagram
$$\xymatrix{
U_1 \ar[r] \ar[d]_f &R_1 \ar[r] \ar[d] &R_2 \ar@{=}[d] \ar@{-->}[r] & \\
X \ar[r] &R_0 \ar[r]  &R_2 \ar@{-->}[r] &
}
$$
where $R_0,R_1,R_2\in \R$. Hence we have the following commutative diagram
$$\xymatrix{
Y \ar[r]^g \ar@{=}[d] &U_1 \ar[r]^f \ar[d]^{u_0} &X\ar[d]^h \ar@{-->}[r] & \\
Y \ar[r] &R_1 \ar[r]_i \ar[d]^{i'} &R_0 \ar@{-->}[r] \ar[d]^{h'} &\\
&R_2 \ar@{-->}[d] \ar@{=}[r] &R_2 \ar@{-->}[d]\\
& &&&}
$$
which implies $Y\in \h$.
\end{proof}

\begin{lem}
In the diagram $(\star)$, we obtain that $g$ is left minimal $\add M$-approximation of $Y$.
\end{lem}

\begin{proof}
Let $y_1:Y \to M_1$ be any morphism where $M_1\in \add M$. We have the following commutative diagram
$$\xymatrix{
\Omega X \ar[d]_b \ar[r]^c &P \ar[d] \ar[r]^p &X \ar@{=}[d] \ar@{-->}[r] & \\
Y \ar[r]^-g \ar@{=}[d] &U_1 \ar[r]^f \ar[d] &X\ar[d]^h \ar@{-->}[r] &\\
Y \ar[r] \ar[d]_{y_0} &I \ar[r] \ar[d] &\Sigma Y \ar[d]^a \ar@{-->}[r] &\\
M \ar[r] &I_M \ar[r]_{i_m} &\Sigma M \ar@{-->}[r] &
}
$$
Since $M$ is $\R$-rigid and $h$ factors through $\R$, we have $ah$ factors through $i_m$, which implies $y_0$ factors through $g$. Hence $g$ is a left $\add M$-approximation of $Y$. Now we show it is also minimal.
If this were not true, then there
would be a decomposition $U_1=U_{11}\oplus U_{12}$ such that
$$g=\svecv{u_1}{0}\colon Y\xrightarrow{~~} U_{11}\oplus U_{12}.$$ It is easy to check that $U_{12}$ is a direct summand of $X$. Hence $X\simeq U_{12}$, a contradiction to what we claimed about $U$ and $X$.
\end{proof}

\begin{lem}
The object $Y$ in $(\star)$ is indecomposable and not in $\add M$.
\end{lem}

\begin{proof}
Let $Y=Y_1\oplus Y_2$. Then $Y_1$ admits an $\EE$-triangle $\xymatrix{Y_1 \ar[r]^-{g_1} &U_1 \ar[r]^-{f_1} &X_1 \ar@{-->}[r] &}$ where $g_1$ is a left minimal $\add U$-approximation of $Y_1$. Then we have the following commutative diagram:
$$\xymatrix@C=1.3cm{
Y_1 \ar[r]^-{g_1} \ar[d]_{\alpha} &U_1 \ar@{=}[d] \ar[r]^-{f_1} &X_1\ar[d]^a \ar@{-->}[r] &\\
Y \ar[r]^-{g} \ar[d]_{\beta} &U_1 \ar[r]^-{f} \ar[d]^{v} &X\ar[d]^b \ar@{-->}[r] &\\
Y_1 \ar[r]^-{g_1} &U_1 \ar[r]^-{f_1} &X_1 \ar@{-->}[r] &
}
$$
where $\beta\alpha=1_{Y_1}$. Since $g_1$ is left minimal, we have that $v$ is an isomorphism. Then $ba$ is also an isomorphism. But $X$ is indecomposable, we have $a$ is an isomorphism. Hence $\alpha$ is an isomorphism. This implies $Y$ is indecomposable.

Now we show $Y\notin \add M$.
If $Y\in \add M$, since $M$ is $\R$-rigid, in the diagram $(\star)$ we have $h$ factors through $i$, hence $1_Y$ factors through $g$. This means the second row splits and X is a direct summand of $U_1$, a contradiction to what we claimed.
\end{proof}



Given an almost maximal $\R$-rigid object, our main result in this
section shows that starting with a complement, we can calculate the other one by
an exchange $\EE$-triangle, which is constructed from a left approximation or a right
approximation.

\begin{thm}\label{main2}
Let $M=U\oplus X$ be a basic maximal $\R$-rigid object where $X$ is a non-projective indecomposable object. Then we have the following.
\begin{itemize}
\item[(a)] If $X\in U*[R]$, there is an $\EE$-triangle $$\xymatrix{Y \ar[r] &U_1 \ar[r]^f &X \ar@{-->}[r] &}$$
where $f$ is a right minimal $\add U$-approximation of $X$ and $Y$ is another complement to $U$ and $U\oplus Y>M$.

\item[(b)] If $X\notin U*[R]$, there is an $\EE$-triangle $$\xymatrix{X \ar[r]^g &U_2 \ar[r] &Y \ar@{-->}[r] &}$$
where $g$ is a left minimal $\add U$-approximation of $X$ and $Y$ is another complement to $U$ and $U\oplus Y<M$.
\end{itemize}
\end{thm}

\begin{proof}
According to the previous lemmas in this section, we first show $U\oplus Y$ is $\R$-rigid.

Consider a morphism $\alpha: U\xrightarrow{} R \xrightarrow{} \Sigma Y$, in the diagram $(\star)$, we obtain that $h' \alpha$ factors through $i'$. So there is a morphism $u:U \to I$ such that $h' \alpha=i'u=h'iu$. Hence we have a morphism $u': U\to X$ such that $\alpha=iu+hu'$. Since $u'$ factors through $f$, we have that $hu'$ also factors through $i$. This implies $\overline {[R]}(U,\Sigma Y)=[R](U,\Sigma Y)$.

Consider a morphism $\beta: Y\xrightarrow{~b_1~} R \xrightarrow{~b_2~} \Sigma U$, since in the diagram $(\star)$, we have $h:X \xrightarrow{~h_1~} R\xrightarrow{~h_2~} \Sigma Y$, we can get the following commutative diagrams:
$$\xymatrix{
\Omega X \ar[r]^q \ar[d]_{h'} &P \ar[r]^p \ar[d] &X \ar@{=}[d] \ar@{-->}[r] &\\
Y \ar[r]^-{g} \ar@{=}[d] &U_1 \ar[r]^-{f} \ar[d]^{u_0} &X\ar[d]^h \ar@{-->}[r] &\\
Y \ar[r] &I \ar[r]_i &\Sigma Y \ar@{-->}[r] &}
$$
and
$$\xymatrix{
\Omega X \ar[r]^q \ar[d]_{h_1'} &P \ar[d] \ar[r]^p &X \ar[d]^{h_1} \ar@{-->}[r] &\\
\Omega R \ar[r]^{p_R} \ar[d]_{h_2'} &P_R \ar[d] \ar[r] &R \ar[d]^{h_2} \ar@{-->}[r] &\\
Y \ar[r]^j  &I \ar[r]  &\Sigma Y \ar@{-->}[r] &\\
}
$$
we have $h'-h_2'h_1'$ factors through $q$. There exists a morphism $p_Y:P\to Y$ such that $p_Yq=h'-h_2'h_1'$. Since $b_1h_2'$ factors through $p_R$, we have that $b_1h_2'h_1'$ factors through $q$. Hence $b_1h'$ factors through $q$. Thus there is a morphism $v: U_1\to R$ such that $b_1=v g$. Since $U$ is $\R$-rigid, $b_2 v$ factors through $i_U$: $\xymatrix{U \ar[r] &I_U \ar[r]^{i_U} &\Sigma U \ar@{-->}[r] &}$. Hence $\beta$ factors through $i_U$ and we get $\overline {[R]}(Y,\Sigma U)=[R](Y,\Sigma U)$.

By the same method and by the fact $\overline {[R]}(U,\Sigma Y)=[R](U,\Sigma Y)$, we can get $\overline {[R]}(Y,\Sigma Y)=[R](Y,\Sigma Y)$.

Now according to the previous results of this section, Theorem \ref{main} and \cite[Theorem 2.17]{AIR}, $Y$ is another complement to $U$, which means $U\oplus Y$ is also maximal $\R$-rigid.

We show $U\oplus Y>M$. If it is not the case, we have $U\oplus Y<M$. But then $X\notin \Fac U$, a contradiction to the assumption that $X\in U*[R]$.
By the same method we can show $\overline {[R]}(Y,\Sigma Y)=[R](Y,\Sigma Y)$.
\medskip

(2) Let $Y$ be another complement to $U$. If $X\notin U*[R]$, then $X\notin \Fac U$. By \cite[Definition-Proposition 2.26]{AIR} and Proposition \ref{mn} we have $U\oplus Y<M$ and $Y\in U*[R]$. By using (1) and the previous lemmas in this section, we get what we want.
\end{proof}

Our main results seem to be new phenomenon when it is applied to exact categories.
In particular, since module categories and triangulated categories can be viewed
as extriangulated categories, our these results generalize their work by
Yang-Zhu \cite{YZ} and Gei{\ss}-Leclerc-Schr\"{o}er \cite{GLS}. Moreover, our proof is not far from the usual module or triangulated case.

\section{Example}

In this section, we give an example illustrating our main results.

\begin{exm}\label{ex1}
Let $\Lambda$ be the $k$-algebra given by the quiver
$$
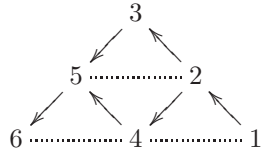
\xymatrix@C=0.4cm@R0.4cm{
&&3 \ar[dl]\\
&5 \ar[dl] \ar@{.}[rr] &&2 \ar[dl] \ar[ul]\\
6 \ar@{.}[rr] &&4 \ar[ul] \ar@{.}[rr] &&1 \ar[ul]}$$
with mesh relations. The Auslander-Reiten quiver of $\B:=\mod\Lambda$ is given by
$$
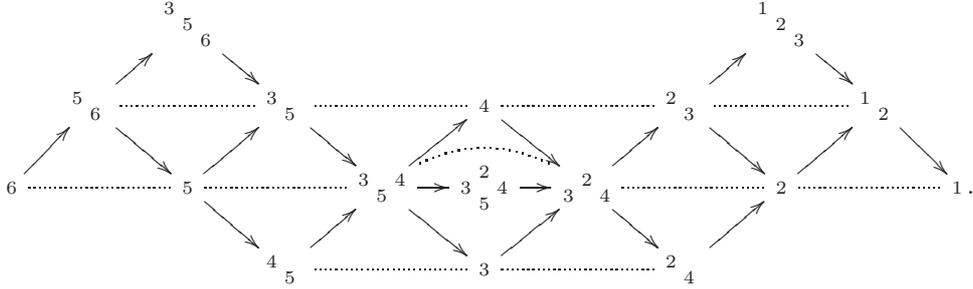
\xymatrix@C=0.4cm@R0.4cm{
&&{\begin{smallmatrix}
3&&\\
&5&\\
&&6
\end{smallmatrix}} \ar[dr] &&&&&&{\begin{smallmatrix}
1&&\\
&2&\\
&&3
\end{smallmatrix}} \ar[dr]\\
&{\begin{smallmatrix}
5&&\\
&6&
\end{smallmatrix}} \ar[ur] \ar@{.}[rr] \ar[dr] &&{\begin{smallmatrix}
3&&\\
&5&
\end{smallmatrix}} \ar@{.}[rr] \ar[dr] &&{\begin{smallmatrix}
4
\end{smallmatrix}} \ar@{.}[rr] \ar[dr] &&{\begin{smallmatrix}
2&&\\
&3&
\end{smallmatrix}} \ar[ur] \ar@{.}[rr] \ar[dr] &&{\begin{smallmatrix}
1&&\\
&2&
\end{smallmatrix}} \ar[dr]\\
{\begin{smallmatrix}
6
\end{smallmatrix}} \ar[ur] \ar@{.}[rr] &&{\begin{smallmatrix}
5
\end{smallmatrix}} \ar[ur] \ar@{.}[rr] \ar[dr] &&{\begin{smallmatrix}
3&&4\\
&5&
\end{smallmatrix}} \ar[ur] \ar[r] \ar[dr] \ar@{.}@/^15pt/[rr] &{\begin{smallmatrix}
&2&\\
3&&4\\
&5&
\end{smallmatrix}} \ar[r] &{\begin{smallmatrix}
&2&\\
3&&4
\end{smallmatrix}} \ar[ur] \ar@{.}[rr] \ar[dr] &&{\begin{smallmatrix}
2
\end{smallmatrix}} \ar[ur] \ar@{.}[rr] &&{\begin{smallmatrix}
1
\end{smallmatrix}}.\\
&&&{\begin{smallmatrix}
4&&\\
&5&
\end{smallmatrix}} \ar[ur] \ar@{.}[rr] &&{\begin{smallmatrix}
3
\end{smallmatrix}} \ar[ur] \ar@{.}[rr] &&{\begin{smallmatrix}
2&&\\
&4&
\end{smallmatrix}} \ar[ur]
}$$
We denote by ``~$\circ$" in the Auslander-Reiten quiver the indecomposable objects belong to a subcategory and by ``~$\bullet$'' the indecomposable objects do not belong to it.
$$
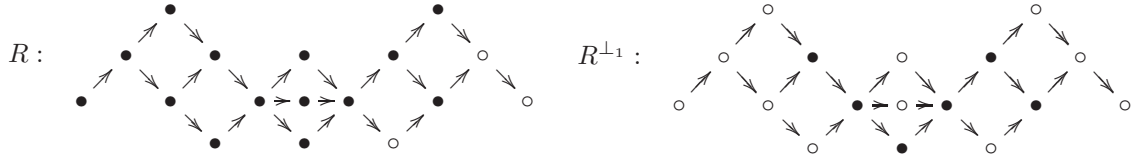
\xymatrix@C=0.2cm@R0.2cm{
&&&\bullet \ar[dr] &&&&&&\bullet \ar[dr]\\
{R:} &&\bullet \ar[ur]  \ar[dr] &&\bullet  \ar[dr] &&\bullet  \ar[dr] &&\bullet  \ar[ur]  \ar[dr] &&\circ \ar[dr]\\
&\bullet \ar[ur]  &&\bullet \ar[ur]  \ar[dr] &&\bullet \ar[ur] \ar[r] \ar[dr] &\bullet \ar[r] &\bullet \ar[ur] \ar[dr] &&\bullet \ar[ur] &&\circ\\
&&&&\bullet \ar[ur] &&\bullet \ar[ur] &&\circ \ar[ur]
\\} \quad
\xymatrix@C=0.2cm@R0.2cm{
&&&\circ \ar[dr] &&&&&&\circ \ar[dr]\\
{R^{\bot_1}:} &&\circ \ar[ur]  \ar[dr] &&\bullet  \ar[dr] &&\circ  \ar[dr] &&\bullet  \ar[ur]  \ar[dr] &&\circ \ar[dr]\\
&\circ \ar[ur]  &&\circ \ar[ur]  \ar[dr] &&\bullet \ar[ur] \ar[r] \ar[dr] &\circ \ar[r] &\bullet \ar[ur] \ar[dr] &&\bullet \ar[ur] &&\circ\\
&&&&\circ \ar[ur] &&\bullet \ar[ur] &&\circ \ar[ur]
}$$
where $R$ is rigid and $R^{\bot_1}=\{X\in\B~|~\Ext^1(R,X)=0\}$.
The quiver of $\overline \h$ is the following:

$$
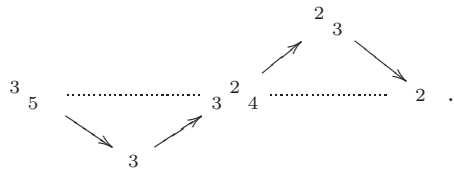
\xymatrix@C=0.4cm@R0.4cm{
&&&{\begin{smallmatrix}
2&\ \\
&3
\end{smallmatrix}}\ar[dr]\\
{\begin{smallmatrix}
3&&\ \\
&5&
\end{smallmatrix}}\ar[dr] \ar@{.}[rr]
&&{\begin{smallmatrix}
&2&\ \\
3&&4
\end{smallmatrix}}\ar[ur] \ar@{.}[rr]
&&{\begin{smallmatrix}
\ &2&\
\end{smallmatrix}}.\\
&{\begin{smallmatrix}
\ &3&\
\end{smallmatrix}} \ar[ur]}$$
It is equivalent to $\mod (kQ/ \langle \beta \alpha \rangle)$, where $Q$ is the quiver $1\xrightarrow{\alpha} 2\xrightarrow{\beta} 3$. According to \cite[Example 6.4]{AIR}, there are $12$ basic support $\tau$-tilting pairs in $\overline \h$, we list them and the maximal $\R$-rigid objects in $\h\subseteq\mod\Lambda$ which are correspondent to them by Theorem \ref{main} below:
\vspace{0.5mm}

\begin{eqnarray*}
(\text{ }0,{\begin{smallmatrix}
3&&\ \\
&5&
\end{smallmatrix}} \oplus {\begin{smallmatrix}
\ &3&\
\end{smallmatrix}} \oplus {\begin{smallmatrix}
2&\ \\
&3
\end{smallmatrix}}) &\longmapsto &R   \nonumber \\[1mm]
({\begin{smallmatrix}
&2&\ \\
3&&4
\end{smallmatrix}}, {\begin{smallmatrix}
3&&\ \\
&5&
\end{smallmatrix}} \oplus {\begin{smallmatrix}
2&\ \\
&3
\end{smallmatrix}})  & \longmapsto &  {\begin{smallmatrix}
&2&\ \\
3&&4
\end{smallmatrix}} \oplus {\begin{smallmatrix}
2&&\\
&4&
\end{smallmatrix}} \oplus {\begin{smallmatrix}
\ &1&\
\end{smallmatrix}}   \nonumber \\[1mm]
( {\begin{smallmatrix}
2&\ \\
&3
\end{smallmatrix}}\oplus {\begin{smallmatrix}
\ &2&\
\end{smallmatrix}}, {\begin{smallmatrix}
3&&\ \\
&5&
\end{smallmatrix}})   & \longmapsto&  {\begin{smallmatrix}
2&\ \\
&3
\end{smallmatrix}}\oplus {\begin{smallmatrix}
\ &2&\
\end{smallmatrix}} \oplus {\begin{smallmatrix}
2&\ \\
&4
\end{smallmatrix}}=N   \nonumber\\[1mm]
({\begin{smallmatrix}
\ &3&\
\end{smallmatrix}} \oplus {\begin{smallmatrix}
&2&\ \\
3&&4
\end{smallmatrix}}, {\begin{smallmatrix}
2&\ \\
&3
\end{smallmatrix}}) & \longmapsto&   {\begin{smallmatrix}
\ &3&\
\end{smallmatrix}} \oplus {\begin{smallmatrix}
&2&\ \\
3&&4
\end{smallmatrix}} \oplus {\begin{smallmatrix}
\ &1&\
\end{smallmatrix}} \nonumber\\[1mm]
({\begin{smallmatrix}
&2&\ \\
3&&4
\end{smallmatrix}} \oplus {\begin{smallmatrix}
2&\ \\
&3
\end{smallmatrix}},{\begin{smallmatrix}
3&\ \\
&5
\end{smallmatrix}})& \longmapsto&  {\begin{smallmatrix}
&2&\ \\
3&&4
\end{smallmatrix}} \oplus {\begin{smallmatrix}
2&\ \\
&3
\end{smallmatrix}} \oplus {\begin{smallmatrix}
2&\ \\
&4
\end{smallmatrix}}  \nonumber\\[1mm]
({\begin{smallmatrix}
\ &3&\
\end{smallmatrix}} \oplus {\begin{smallmatrix}
2&\ \\
&3
\end{smallmatrix}}\oplus {\begin{smallmatrix}
\ &2&\
\end{smallmatrix}},0 \text{ })& \longmapsto&  {\begin{smallmatrix}
\ &3&\
\end{smallmatrix}} \oplus {\begin{smallmatrix}
2&\ \\
&3
\end{smallmatrix}}\oplus {\begin{smallmatrix}
\ &2&\
\end{smallmatrix}}\nonumber\\[1mm]
({\begin{smallmatrix}
\ &2&\
\end{smallmatrix}}, {\begin{smallmatrix}
3&&\ \\
&5&
\end{smallmatrix}} \oplus {\begin{smallmatrix}
\ &3&\
\end{smallmatrix}}) &\longmapsto & {\begin{smallmatrix}
2&&\\
&4&
\end{smallmatrix}} \oplus {\begin{smallmatrix}
\ &2&\
\end{smallmatrix}} \oplus {\begin{smallmatrix}
1&&\\
&2&
\end{smallmatrix}}\nonumber=M \\[1mm]
({\begin{smallmatrix}
3&&\ \\
&5&
\end{smallmatrix}}, {\begin{smallmatrix}
\ &3&\
\end{smallmatrix}} \oplus {\begin{smallmatrix}
2&\ \\
&3
\end{smallmatrix}})&\longmapsto & {\begin{smallmatrix}
3&&\ \\
&5&
\end{smallmatrix}} \oplus {\begin{smallmatrix}
1&&\ \\
&2&
\end{smallmatrix}} \oplus {\begin{smallmatrix}
\ &1&\
\end{smallmatrix}}\nonumber \\[1mm]
({\begin{smallmatrix}
3&\ \\
&5
\end{smallmatrix}}\oplus {\begin{smallmatrix}
\ &2&\
\end{smallmatrix}}, {\begin{smallmatrix}
\ &3&\
\end{smallmatrix}}) &\longmapsto &{\begin{smallmatrix}
3&\ \\
&5
\end{smallmatrix}}\oplus {\begin{smallmatrix}
\ &2&\
\end{smallmatrix}} \oplus {\begin{smallmatrix}
1&\ \\
&2
\end{smallmatrix}}\nonumber \\[1mm]
 ({\begin{smallmatrix}
3&\ \\
&5
\end{smallmatrix}} \oplus {\begin{smallmatrix}
\ &3&\
\end{smallmatrix}}, {\begin{smallmatrix}
2&\ \\
&3
\end{smallmatrix}})&\longmapsto &{\begin{smallmatrix}
3&\ \\
&5
\end{smallmatrix}} \oplus {\begin{smallmatrix}
\ &3&\
\end{smallmatrix}} \oplus {\begin{smallmatrix}
\ &1&\
\end{smallmatrix}}\nonumber \\[1mm]
({\begin{smallmatrix}
3&\ \\
&5
\end{smallmatrix}} \oplus {\begin{smallmatrix}
2&\ \\
&3
\end{smallmatrix}}\oplus {\begin{smallmatrix}
\ &2&\
\end{smallmatrix}},0 \text{ }) &\longmapsto & {\begin{smallmatrix}
3&\ \\
&5
\end{smallmatrix}} \oplus {\begin{smallmatrix}
2&\ \\
&3
\end{smallmatrix}}\oplus {\begin{smallmatrix}
\ &2&\
\end{smallmatrix}}\nonumber \\[1mm]
({\begin{smallmatrix}
3&\ \\
&5
\end{smallmatrix}} \oplus {\begin{smallmatrix}
\ &3&\
\end{smallmatrix}} \oplus {\begin{smallmatrix}
2&\ \\
&3
\end{smallmatrix}},0 \text{ })&\longmapsto &{\begin{smallmatrix}
3&\ \\
&5
\end{smallmatrix}} \oplus {\begin{smallmatrix}
\ &3&\
\end{smallmatrix}} \oplus {\begin{smallmatrix}
2&\ \\
&3
\end{smallmatrix}}
\end{eqnarray*}
\medskip

Let $\Lambda \oplus {\begin{smallmatrix}
\ &2&\
\end{smallmatrix}} \oplus {\begin{smallmatrix}
2&\ \\
&4
\end{smallmatrix}}=U$. Then we can find that ${\begin{smallmatrix}
2&\ \\
&3
\end{smallmatrix}} \notin U*[R]$. According to Theorem \ref{main2}, we can find another complement in the following short exact sequence $${\begin{smallmatrix}
2&\ \\
&3
\end{smallmatrix}} \to {\begin{smallmatrix}
1&&\\
&2&\\
&&3
\end{smallmatrix}} \oplus {\begin{smallmatrix}
\ &2&\
\end{smallmatrix}} \to {\begin{smallmatrix}
1&\ \\
&2
\end{smallmatrix}}.$$ We also have $U\oplus {\begin{smallmatrix}
1&\ \\
&2
\end{smallmatrix}}=\Lambda \oplus M< U \oplus {\begin{smallmatrix}
2&\ \\
&3
\end{smallmatrix}}=\Lambda \oplus N.$
\end{exm}

\end{document}